\documentclass[12pt]{amsart}
\usepackage{amsmath,amssymb,amsfonts,amsthm,amsopn}
\usepackage{latexsym,graphicx}
\usepackage{xcolor}
\usepackage{color, colortbl}
\usepackage{tikz-cd}
\usepackage{enumitem}
\usepackage{pb-diagram}
\usepackage[title]{appendix}

\newtheorem{theorem}{Theorem}[section]
\newtheorem{lemma}[theorem]{Lemma}
\newtheorem{corollary}[theorem]{Corollary}
\newtheorem{proposition}[theorem]{Proposition}
\newtheorem{definition}[theorem]{Definition}

\newtheorem{example}[theorem]{Example}
\newtheorem{remark}[theorem]{Remark}

\newcommand{\beqa}{\begin{eqnarray*}}
\newcommand{\eeqa}{\end{eqnarray*}}

\DeclareMathOperator*{\Sp}{\mbox{Sp}}
\DeclareMathOperator*{\Shp}{\mbox{Shp}}
\DeclareMathOperator*{\Mp}{\mbox{Mp}}
\DeclareMathOperator*{\Sym}{\mbox{Sym}}
\DeclareMathOperator*{\GL}{\mbox{GL}}

\newcommand{\field}[1]{\mathbb{#1}}
\newcommand{\bR}{\field{R}}        
\newcommand{\bC}{\field{C}}        
\def\la{\lambda}

\def\cF{\mathcal{F}}              
\def\cS{\mathcal{S}}
\def\cD{\mathcal{D}}

\def\cK{\mathcal{K}}

\def\cA{\mathcal{A}}
\def\cJ{\mathcal{J}}
\def\cI{\mathcal{I}}

\def\rd{\bR^d}

\def\rdd{{\bR^{2d}}}

\def\R{\right)}

\def\<{\left<}
\def\>{\right>}

\def\mv1{M_v^1}

\def\mn{(m,n)}
\def\mn'{(m',n')}

\newcommand{\norm}[1]{\lVert#1\rVert}

\def\R{\mathbb{R}}
\def\Ren{\mathbb{R}^d}

\def\f{\varphi}

\def\Sn2{S_{2}(L^{2}(\Ren))}
\def\S1{S_{1}(L^{2}(\Ren))}
\def\sig00{\sigma_{0,0}}

\def\la{\langle}
\def\ra{\rangle}

\begin{document}

\begin{abstract} 
	Housdorff-Young's inequality establishes the boundedness of the Fourier transform from $L^p$ to $L^q$ spaces for $1\leq p\leq2$ and $q=p'$, where $p'$ denotes the Lebesgue-conjugate exponent of $p$. This paper extends this classical result by characterizing the $L^p-L^q$ boundedness of all metaplectic operators, which play a significant role in harmonic analysis. We demonstrate that metaplectic operators are bounded on Lebesgue spaces if and only if their symplectic projection is either free or lower block triangular. As a byproduct, we identify metaplectic operators that serve as homeomorphisms of $L^p$ spaces. 
	To achieve this, we leverage a parametrization of the symplectic group by F. M. Dopico and C. R. Johnson involving products of complex exponentials with quadratic phase, Fourier multipliers, linear changes of variables, and partial Fourier transforms. Then, we use our findings to provide boundedness results within $L^p$ spaces for pseudodifferential operators with symbols in Lebesgue spaces, and quantized by means of metaplectic operators. These quantizations consists of shift-invertible metaplectic Wigner distributions, which play a fundamental role in measuring local phase-space concentration of signals. Using the Dopico-Johnson factorization, we infer a decomposition law for metaplectic operators on $L^2(\rdd)$ in terms of shift-invertible metaplectic operators, establish the density of shift-invertible symplectic matrices in $\Sp(2d,\bR)$, and prove that the lack of shift-invertibility prevents metaplectic Wigner distributions to define the so-called modulation spaces $M^p(\rd)$. 
\end{abstract}

\title[$L^p$ boundedness of the metaplectic representation]{Boundedness of metaplectic operators within $L^p$ spaces, applications to pseudodifferential calculus, and time-frequency representations}

\author{Gianluca Giacchi}
\address{Universit\'a di Bologna, Dipartimento di Matematica, Piazza di Porta San Donato 5, 40126 Bologna, Italy; University of Lausanne, Switzerland; HES-SO School of Engineering, Rue De L'Industrie 21, Sion, Switzerland; Centre Hospitalier Universitaire Vaudois, Switzerland}
\email{gianluca.giacchi2@unibo.it}

\thanks{}

\maketitle

\section{Introduction} 

The metaplectic group $\Mp(d,\bR)$ appeared in mathematics in the second half of 20th century. Initially explored by L. C. P. Van Hove in his Ph.D. thesis \cite{van1951certaines}, it was later reintroduced in 1959 by I. E. Segal \cite{osti4250939} and in 1962 by D. Shale \cite{shale1962linear}, within the framework of quantum mechanics. Subsequently, A. Weyl extended its study to the realm of number theory in 1964 \cite{weil1964certains}.

Algebraically, the metaplectic group $\Mp(d,\bR)$ is a realization of the double cover of the symplectic group $\Sp(d,\bR)$. From a mathematical analysis perspective, a metaplectic operator $\hat S\in \Mp(d,\bR)$ is a unitary operator on $L^2(\rd)$ satisfying the intertwining relation: 
\[
	\hat S\rho(x,\xi;\tau)\hat S^{-1}=\rho(S(x,\xi);\tau), \qquad x,\xi\in\rd, \ \tau\in\bR,
\]
where $\rho(x,\xi;\tau)g(t)=e^{2\pi i\tau}e^{-i\pi x\cdot \xi}e^{2\pi it\cdot\xi}g(t-x)$, $g\in L^2(\rd)$, is the Schr\"odinger representation of the Heisenberg group. 

Many aspects of harmonic analysis, such as frame theory \cite{feichtinger2008metaplectic, cordero2024metaplectic}, quantum mechanics \cite{de1998quantum, gosson2005weyl}, PDEs \cite{hormander2007analysis, hormander1995symplectic} and Schr\"odinger equations \cite{folland1989harmonic}, time-frequency analysis \cite{cordero2020time, grochenig2013foundations}, can be settled in the framework of metaplectic operators. Despite the very algebraic definition of the metaplectic group, any metaplectic operator $\hat S$ reduces to the composition of a few concrete operators, revealing a manageable structure for $\Mp(d,\bR)$: the Fourier transform
\begin{equation}\label{intro1}
	\cF f(\xi)=\int_{\rd}f(x)e^{-2\pi ix\cdot\xi}dx, \qquad f\in\cS(\rd),
\end{equation}
the products by chirps:
\begin{equation}\label{intro2}
	\mathfrak{p}_Cf(x)=e^{i\pi Cx\cdot x}f(x), \qquad f\in L^2(\rd),
\end{equation}
($C\in\bR^{d\times d}$ symmetric), and the unitary linear changes of variables:
\begin{equation}\label{intro3}
	\mathfrak{T}_Lf(x)=|\det(L)|^{1/2}f(Lx), \qquad f\in L^2(\rd),
\end{equation}
($L\in\bR^{d\times d}$ invertible) generate the group $\Mp(d,\bR)$, \cite{folland1989harmonic}. However, the properties of metaplectic operators are not always evident by their factorization in terms of the generators of $\Mp(d,\bR)$, and they may depend on how the operators in \eqref{intro1}-\eqref{intro3} are combined. Nevertheless, during the years, many different factorizations have been established, facilitating the study of metaplectic operators, according to the context. In this work, we use a parametrization of the symplectic group which is due to F. M. Dopico and C. R. Johnson \cite[Theorem 3.2]{dopico2009parametrization} to factorize metaplectic operators. Specifically, if $\hat S\in\Sp(d,\bR)$, then there exist matrices $P,Q\in\bR^{d\times d}$ symmetric, $L\in\bR^{d\times d}$ invertible and indices $\cJ\subseteq\{1,\ldots,d\}$, such that
\begin{equation}\label{intro4}
	\hat S=\mathfrak{p}_Q\mathfrak{T}_L\mathfrak{m}_P\cF_\cJ,
\end{equation}
up to a sign, where 
\begin{equation}\label{intro41}
	\mathfrak{m}_Pf(x)=\cF^{-1}(e^{i\pi Pu\cdot u}\hat f)(x), \qquad f\in\cS(\rd),
\end{equation}
and $\cF_\cJ$ is the partial Fourier transform with respect to the variables indexed by $\cJ$, see Section \ref{sec:prelim} below. On the other side, their feature of being factorized by the operators in \eqref{intro1}-\eqref{intro3} does not limit the variability of the applications of metaplectic operators, which exhibit variegated behavior in the contexts. A prototypical example is provided by recent developments in the theory of time-frequency representations. In \cite{cordero2022wigner, cordero2022wigner2, giacchi2022metaplectic} the authors introduce a generalization of the cross-Wigner distribution,
\begin{equation}\label{intro5}
	W(f,g)(x,\xi)=\int_{\rd}f(x+t/2)\overline{g(x-t/2)}e^{-2\pi i t\cdot\xi}dt, \qquad x,\xi\in\rd,
\end{equation}
\cite{wigner1932quantum, cohen1966generalized, leon1995time} using metaplectic operators. Properties such as covariance, belonging to the Cohen's class, being generalized spectrograms, and the feature of measuring local time-frequency content were characterized in \cite{cordero2023characterization, cordero2023symplectic, cordero2024metaplectic, cordero2024unified} in terms of the block structure of symplectic projections. Moreover, a recent important contribution by H. Fuhr and I. Shafkulovska, c.f. \cite{fuhr2024metaplectic}, characterizes the boundedness of metaplectic operators on the so-called \emph{modulation spaces}, in the Banach setting. For given $1\leq p,q\leq\infty$ and $g\in\cS(\rd)\setminus\{0\}$ fixed, the modulation space $M^{p,q}(\rd)$ is the space of tempered distributions $f\in\cS'(\rd)$ such that $W(f,g)\in L^{p,q}(\rdd)$, where $L^{p,q}(\rdd)$ are the mixed-norm Lebesgue spaces. We write $M^p(\rd)=M^{p,p}(\rd)$ if $p=q$. Notably, $\hat S:M^{p}(\rd)\to M^{p}(\rd)$ for every $1\leq p\leq\infty$ and $\hat S:M^{p,q}(\rd)\to M^{p,q}(\rd)$ for every $1\leq p\neq q\leq\infty$ if and only if the projection $S\in\Sp(d,\bR)$ has block decomposition:
\begin{equation}\label{intro6}
	S=\begin{pmatrix}
		A & B\\
		C & D
	\end{pmatrix}, \qquad A,B,C,D\in\bR^{d\times d},
\end{equation}
with $C=0$ (the matrix with all zero entries). Consequently, metaplectic operators exhibit optimal boundedness properties on modulation spaces $M^p(\rd)$. This is a concrete example of how a property of metaplectic operators, or related objects, can be inferred by the structure of the related projections on the symplectic group. 

In this paper, we focus on the boundedness of metaplectic operators on Lebesgue spaces $L^p(\rd)$. The fact that metaplectic operators do not behave on Lebesgue spaces as well as they do on modulation spaces shall not surprise the reader. For example, it is well known that the Fourier transform is bounded from $L^p(\rd)$ to $L^{q}(\rd)$ if and only if $1\leq p\leq 2$ and $q=p'$ is the Lebesgue conjugate exponent of $p$. In these instances, the operator norm of $\cF$ was determined in 1975 by W. Beckner, c.f. \cite{beckner1975inequalities}, as $\norm{\cF}_{B(L^p,L^{p'})}=({p^{1/p}}/{(p')^{1/p'}})^{d/2}$, for $1\leq p\leq2$. This result generalizes to metaplectic operators $\hat S$ with symplectic projections $S$ having block decompositions \eqref{intro6} satisfying $\det(B)\neq0$. In 1960, L. H\"ormander proved that metaplectic multipliers \eqref{intro41}, cannot be bounded from $L^p(\rd)$ to itself unless $p=2$ or $P=0$, as detailed in \cite[Lemma 1.4]{hormander1960estimates}. If $P$ is invertible, a direct consequence of the Riesz-Thorin interpolation theorem establishes their boundedness from $L^p(\rd)$ to $L^{q}(\rd)$ if $1\leq p\leq2$ and $q=p'$.

The main contribution of this work is three-fold. First, we characterize boundedness of metaplectic operators within Lebesgue spaces. 
\begin{theorem}\label{thmintro1}
	Let $\hat S\in \Mp(d,\bR)$ have projection $S$ with block decomposition \eqref{intro6}. Then: \\ 
	(i) if $B=0$, then $\hat S$ is a surjective quasi-isometry of $L^p(\rd)$, with $\norm{\hat S}_{B(L^p)}=|\det(L)|^{1/2-1/p}$.\\
	(ii) If $B$ is invertible, then $\hat S:L^p(\rd)\to L^q(\rd)$ is bounded if and only if $1\leq p\leq2$ and $q=p'$.\\
	(iii) If $B\neq0$ is not invertible, then $\hat S:L^p(\rd)\to L^q(\rd)$ is not bounded for any $0<p,q\leq\infty$, $p,q\neq2$.
\end{theorem}

The main tool to prove this result involves an intertwining relation which emphasises the interaction between partial Fourier transforms $\cF_\cJ$ and Fourier multipliers $\mathfrak{m}_P$, defined above. 

\begin{lemma}
	Let $P\in\bR^{d\times d}$ be symmetric and $\cJ\subseteq\{1,\ldots,d\}$. Let $\cJ^c=\{1,\ldots,d\}\setminus\cJ$. Then,
	\begin{equation}\label{intro8}
		\mathfrak{m}_P\cF_\cJ=\cF_\cJ\mathfrak{m}_{I_{\cJ^c}PI_{\cJ^c}}(\mathfrak{T}_{I+I_{\cJ^c} P I_\cJ}\mathfrak{p}_{-I_\cJ P I_{\cJ}}),
	\end{equation}
	where $\cF_\cJ$ is the partial Fourier transform with respect to the variables indexed by $\cJ$, the other operators appearing in \eqref{intro8} are defined as in \eqref{intro2}, \eqref{intro3} and \eqref{intro41}, and $I_{\cJ}$ is the matrix of the projection $x\in\rd\mapsto\{x\in\rd: x_j=0 \ \forall j\notin\cJ\}$. The definition of $I_{\cJ^c}$ is analogous.
\end{lemma}
Observe that the operator $\mathfrak{T}_{I+I_{\cJ^c} P I_\cJ}\mathfrak{p}_{-I_\cJ P I_{\cJ}}$ appearing in \eqref{intro8} is a homeomorphism of $L^p(\rd)$ for every $p$. So, the very core of right hand-side of \eqref{intro8} is the operator $\cF_\cJ\mathfrak{m}_{I_{\cJ^c}PI_{\cJ^c}}$, where the contributions of $\cJ$ and $\cJ^c$ appear separated.

For a given metaplectic operator $\hat\cA\in \Mp(2d,\bR)$, it is possible to construct a quantization that generalizes the Wigner distribution \eqref{intro5}, by considering:
\begin{equation}\label{intro9}
	W_\cA(f,g)=\hat\cA(f\otimes\bar g), \qquad f,g\in\cS(\rd)
\end{equation}
(metaplectic Wigner distribution). Consequently, the pseudodifferential operator $Op_\cA(a):\cS(\rd)\to\cS'(\rd)$ with symbol $s\in\cS'(\rdd)$ and quantization $W_\cA$ is defined as:
\[
	\la Op_\cA(a)f,g\ra=\la a,W_\cA(g,f)\ra, \qquad f,g\in\cS(\rd),
\]
An important class of quantizations \eqref{intro9} includes shift-invertible metaplectic Wigner distributions, defined in \cite{cordero2023characterization} and characterized in \cite{cordero2024unified} as:
\begin{equation}\label{intro10}
	W_\cA(f,g)(z)=|\det(L)|^{1/2}\Phi_C(Lz)W(f,\hat Sg)(Lz), \qquad f,g\in L^2(\rd), \ z\in\rdd,
\end{equation}
for some $L\in\GL(2d,\bR)$, $C\in\Sym(2d,\bR)$ and $\hat S\in\Mp(d,\bR)$. The boundedness of these bilinear operators within Lebesgue spaces depends on the choice of $\hat S$. The second contribution of this work is using Theorem \ref{thmintro1} to improve \cite[Proposition 3.6]{cordero2024unified}:

\begin{theorem}
Let $1\leq p,q\leq\infty$. Let $W_\cA$ be as in \eqref{intro10}, with $S$ having block decomposition \eqref{intro6}. Let $a\in L^q(\rdd)$ and $Op_\cA(a):\cS(\rd)\to\cS'(\rd)$ be the associated metaplectic operator. The following statements hold true.\\
	(i) If $B=0$, then $Op_\cA(a)\in B(L^p(\rd))$ if and only if $q\leq 2$ and $q\leq p\leq q'$.\\
	(ii) If $\det(B)\neq0$, $1\leq q\leq 2$ and $q\leq p\leq q'$, then $Op_\cA(a)\in B(L^p(\rd), L^{p'}(\rd))$.
\end{theorem}

The third and last contribution of this work is in the field of time-frequency representations. First, we prove a density result for shift-invertible metaplectic Wigner distributions.

\begin{theorem}\label{thmintro3}
	The following statements hold true.\\
	(i) The space $\Shp(2d,\bR)$ of shift-invertible symplectic matrices is dense in $\Sp(2d,\bR)$. \\
	(ii) For every $\cA\in \Sp(2d,\bR)$ there exists $\cA'\in\Sp(2d,\bR)$ shift-invertible and $\Xi\in\Sp(2d,\bR)$ free (see Section \ref{sec:prelim} below) such that
	\[
		\cA=\Xi\cA'.
	\]
\end{theorem}

The Rihacek distribution $W_0(f,g)(x,\xi)=f(x)\overline{\hat g(\xi)}e^{-2\pi i\xi x}$ ($f,g\in\cS(\rd)$) is a non shift-invertible metaplectic Wigner distribution and $\norm{W_0(f,g)}_p\asymp\norm{f}_p$. Stated differently, there exists $W_\cA$ non shift-invertible so that $\norm{f}_{M^p}=\norm{W(f,g)}_p\not\asymp\norm{W_\cA(f,g)}_p$ (where $g\in\cS(\rd)$ is any non-zero function). However, it was still an open question whether this occurs for every non shift-invertible metaplectic Wigner distribution. The last contribution of this work is the following result.

\begin{theorem}\label{thmintro4}
	Let $W_\cA$ be a non shift-invertible metaplectic Wigner distribution, i.e., $W_\cA$ cannot be written in the form \eqref{intro10}. Let $g\in \cS(\rd)\setminus\{0\}$. Then, 
	\[
		\norm{W_\cA(f,g)}_p\not\asymp\norm{f}_{M^p}.
	\]
\end{theorem}

This paper is structured as follows: Section \ref{sec:prelim} outlines the notation and introduces preliminary concepts regarding metaplectic operators. In section \ref{sec:oneD} we synthesize the elementary case of metaplectic operators on $L^2(\bR)$, while subsequent sections explore the multivariate scenario. Section \ref{sec:free} presents a characterization of non-free symplectic matrices and their associated metaplectic operators, drawing on a decomposition by F. M. Dopico and C. R. Johnson. In Section \ref{sec:B0}, we establish that compositions of linear changes of variables with chirp products encompass all metaplectic operators that serve as homeomorphisms of $L^p$. Furthermore, section \ref{sec:unbound} concludes the investigation by addressing the remaining cases and establishing the $L^p$ (un-)boundedness of metaplectic operators. Section \ref{sec:Appl} explores applications to pseudodifferential operators quantized via shift-invertible metaplectic operators. Finally, in Section \ref{sec:attfan} we prove Theorems \ref{thmintro3} and \ref{thmintro4}.
\section{Preliminaries}\label{sec:prelim}
	\subsection{Notation} In this work, we will use the following notation. 
	
	{\bf Linear algebra.} We denote by $xy=x\cdot y$ the standard inner product in $\bR^d$. $\mbox{Sym}(d,\bR)$ is the symmetric group of matrices $d\times d$, and $\mbox{GL}(d,\bR)$ is the group of $d\times d$ invertible matrices. The matrix $I_d$ is the $d\times d$ identity and $0_d$ is the $d\times d$ matrix with all zero entries. When $d$ can be omitted without causing confusion, we write $I$ and $0$, respectively. $\mbox{eig}(P)$ is the set of the eigenvalues of a matrix $P\in\bR^{d\times d}$ and $\mbox{diag}(\mbox{eig}(P))$ denotes the diagonal matrix with diagonal entries given by the eigenvalues of $P$. 
	
	{\bf Index notation.} We will make extensive use of indices. To facilitate the reading, we introduce the following notation. 
	
	We denote set of indices with calligraphic capitals. If $\cJ\subseteq\{1,\ldots,d\}$, $\cJ^c=\{1,\ldots,d\}\setminus\cJ$. Also, $I_\cJ\in\bR^{d\times d}$ denotes the diagonal matrix with $j$-th diagonal entry equal to 1 if $j\in\cJ$ and 0 otherwise. Observe that if $\cJ=\emptyset$, then $I_\cJ=0$. If $\cJ,\cK\subseteq\{1,\ldots,d\}$ and $P=(p_{j,k})_{i,j=1}^d\in \bR^{d\times d}$, then $P_{\cJ\cK}=(p_{j,k})_{j\in\cJ,k\in\cK}$. Moreover, if $x=(x_j)_{j=1}^d\in\rd$ and $\cJ=\{1\leq j_1<\ldots<j_r\leq d\}\subseteq\{1,\ldots,d\}$, we write $x_\cJ=(x_{j_1},\ldots,x_{j_r})\in\bR^r$ and $dx_\cJ=dx_{j_1}\ldots dx_{j_r}$.
	
	{\bf Function spaces.} We denote by $\cS(\rd)$ the Schwartz space of rapidly decreasing smooth functions on $\rd$. Its topological dual $\cS'(\rd)$ is the space of tempered distributions. The sesquilinear inner product of $L^2(\rd)$, i.e., $\la f,g\ra=\int_{\rd}f(x)\overline{g(x)}dx$, $f,g\in L^2(\rd)$, extends uniquely to a duality pairing $\la\cdot,\cdot\ra:(f,g)\in\cS'(\rd)\times\cS(\rd)\to \la f,g\ra=f(\bar g)\in\bC$, which is antilinear in the second component. The \emph{Dirac's delta} (point mass) distribution is the tempered distribution $\delta_0\in\cS'(\rd)$ such that $\la \delta_0,g\ra=\overline{g(0)}$ for every $g\in\cS(\rd)$. If $f,g\in\cS'(\rd)$, $f\otimes g$ denotes their tensor product. Notably, if $f,g$ are functions, $f\otimes g(x,y)=f(x)g(y)$. If $0<p,q\leq\infty$, $B(L^p(\rd),L^q(\rd))$ denotes the space of bounded linear operators from $L^p(\rd)$ to $L^q(\rd)$. We also write $B(L^p(\rd))=B(L^p(\rd),L^p(\rd))$.
	
\subsection{Symplectic group}
	A matrix $S\in \bR^{2d\times2d}$ is \emph{symplectic} if
	\begin{equation}\label{blockS}
		S = \begin{pmatrix}
			A & B\\
			C & D
		\end{pmatrix},
	\end{equation}
	with the blocks $A,B,C,D\in\bR^{d\times d}$ satisfying the following relations:
	\begin{equation}\label{defSp}
		\begin{cases}
			A^TC=C^TA,\\
			B^TD=D^TB,\\
			A^TD-C^TB=I,
		\end{cases}
	\end{equation}
	Observe that if $A\in\GL(d,\bR)$, \eqref{defSp} is equivalent to:
	\[
		\begin{cases}
			CA^{-1}=A^{-T}C^T,\\
			B^TD=D^TB,\\
			D=A^{-T}+CA^{-1}B,
		\end{cases}
	\]
	Furthermore, $S\in\Sp(d,\bR)$ if and only if
	\[
		S^{-1}=\begin{pmatrix}
			D^T & -B^T\\
			-C^T & A^T
		\end{pmatrix}.
	\]
	
	A \emph{free symplectic matrix} is a symplectic matrix $S$ with block decomposition \eqref{blockS} having $\det(B)\neq0$. 
	
	The $2d\times 2d$ symplectic group is denoted by $\mbox{Sp}(d,\bR)$ and it is generated by the symplectic matrix of the standard symplectic form of $\rdd$:
	\begin{equation}\label{defJ}
		J=\begin{pmatrix}
			0 & I \\
			-I & 0
		\end{pmatrix},
	\end{equation}
	and by matrices in the form:
	\begin{equation}\label{defVCDL}
		V_P = \begin{pmatrix}
			I & 0\\
			P & I
		\end{pmatrix}, \qquad \mathcal{D}_L=\begin{pmatrix}
			L^T & 0\\
			0 & L^{-1}
		\end{pmatrix},
	\end{equation}
	where $P\in\mbox{Sym}(d,\R)$ and $L\in\mbox{GL}(d,\bR)$. If $\cJ\subseteq\{1,\ldots,d\}$, we consider the symplectic interchange matrix:
	\begin{equation}\label{sympInter}
		\Pi_\cJ=\begin{pmatrix}
			I_{\cJ^c} & I_\cJ\\
			-I_\cJ & I_{\cJ^c}
		\end{pmatrix}.
	\end{equation}
	Observe that $I_\cJ I_{\cJ^c}=0$, and $I_\cJ^2=I_\cJ$, so that
	\begin{equation}\label{defPiJ}
		\Pi_\cJ^{-1}=\Pi_\cJ^T=\begin{pmatrix}
			I_{\cJ^c} & -I_\cJ\\
			I_\cJ & I_{\cJ^c}
		\end{pmatrix}
	\end{equation} 
	The symplectic group can be parametrized in terms of the matrices in \eqref{defVCDL} and \eqref{sympInter}, as stated in the following result by F. M. Dopico and C. R. Johnson, see \cite[Theorem 3.2]{dopico2009parametrization}.
	\begin{proposition}\label{parametrization}
		Let $S\in\mbox{Sp}(d,\bR)$. Then, there exist (non unique) $P,Q\in\mbox{Sym}(d,\bR)$, $L\in\mbox{GL}(d,\bR)$ and $\cJ\subseteq\{1,\ldots,d\}$ such that:
		\begin{equation}\label{Sdecomp}
			S=V_Q\cD_L V_P^T\Pi_\cJ.
		\end{equation}
	\end{proposition} 
	
	\begin{definition}
		If $S\in\Sp(d,\bR)$. We refer to any factorization of $S$ in the form \eqref{Sdecomp} as to a \emph{Dopico-Johnson factorization} of $S$.
	\end{definition}
	
	\begin{remark}\label{rem-par}
		As observed in \cite{dopico2009parametrization}, Proposition \ref{parametrization} has many formulations, varying according to the order in which the matrices $V_Q$, $\cD_L$, $V_P^T$ and $\Pi_\cJ$ appear in \eqref{Sdecomp}.
	\end{remark}
	
\subsection{Metaplectic operators}
	The Schr\"odinger representation of the Heisenberg group $\rho$ is:
	\[
		\rho(x,\xi;\tau)f(t)=e^{2\pi i\tau}e^{-i\pi \xi x}e^{2\pi i\xi t}f(t-x), \qquad f\in L^2(\rd), \tau\in\bR, x,\xi\in\rd.
	\]
	For every $S\in\mbox{Sp}(d,\bR)$ there exists $\hat S:L^2(\rd)\to L^2(\rd)$ unitary such that:
	\begin{equation}\label{intertS}
	\hat S\rho(z;\tau)\hat S^{-1}=\rho(Sz;\tau), \qquad z=(x,\xi)\in\rdd, \tau\in\bR.
	\end{equation}
	$\hat S$ is called \textit{metaplectic operator}. If $\hat S$ satisfies the intertwining relation \eqref{intertS}, so does every operator in the form $c\hat S$, with $c\in\bC$, $|c|=1$. Nonetheless, the group $\{\hat S:S\in\mbox{Sp}(d,\bR)\}$ has a subgroup, denoted by $\mbox{Mp}(d,\bR)$, containing exactly two metaplectic operators for each symplectic matrix. $\mbox{Mp}(d,\bR)$ is called \emph{metaplectic group}. The projection $\pi^{Mp}:\mbox{Mp}(d,\bR)\to\mbox{Sp}(d,\bR)$ is a group homomorphism with kernel $\ker(\pi^{Mp})=\{\pm Id_{L^2}\}$. This means that $\pi^{Mp}(\hat S_1)=\pi^{Mp}(\hat S_2)$ if and only if $\hat S_1=\hat S_2$ up to a sign.
	
	To facilitate the reading, we transport the terminology from $\Sp(d,\bR)$ to $\Mp(d,\bR)$ and say that a metaplectic operator is \emph{free} if its projection is free.
	
	\begin{proposition}
		Let $\hat S\in\mbox{Mp}(d,\bR)$. Then, $\hat S$ enjoys the following continuity properties:\\
		(i) $\hat S: L^2(\rd)\to L^2(\rd)$ is unitary.\\
		(ii) $\hat S:\cS(\rd)\to\cS(\rd)$ is a homeomorphism.\\
		(iii) $\hat S$ extends to a homeomorphism of $\cS'(\rd)$ as follows:
		\[
			\la \hat Sf,g\ra =\la f,\hat S^{-1}g\ra, \qquad f\in\cS'(\rd), g\in\cS(\rd).
		\]
	\end{proposition}

Examples of metaplectic operators are reported hereafter. In what follows it may be useful to denote some ot the components of a vector $v\in\rd$ with different letters, and the matrices $I_\cJ$ can be used for the purpose: if $x,\xi\in\rd$, $v=I_\cJ \xi+I_{\cJ^c}x$ is the vector with coordinates:
\[
	v_j=\begin{cases}
	\xi_j & \text{if $j\in\cJ$},\\
	x_j & \text{if $j\in\cJ^c$}.
	\end{cases}
\]
\begin{example}\label{esMetap}
	(i) The Fourier transform $\cF$, defined for every $f\in\cS(\rd)$ as
	\[
		(\cF f)(\xi)=\hat f(\xi)=\int_{\rd}f(x)e^{-2\pi i\xi x}dx, \qquad \xi\in\rd,
	\]
	is a metaplectic operator. Moreover, $\pi^{Mp}(\cF)=J$, where $J$ is defined as in \eqref{defJ}.\\
	(ii) More generally, if $\cJ=\{j_1,\ldots,j_r\}\subseteq\{1,\ldots,d\}$, the \emph{partial Fourier transform} with respect to the variables indexed by $\cJ$ is the operator $\cF_\cJ:\cS(\rd)\to\cS(\rd)$ given by:
	\[
		(\cF_\cJ f )(I_\cJ\xi+ I_{\cJ^c}x)=\int_{\bR^r}f(x)e^{-2\pi i\sum_{j\in\cJ}\xi_j x_j}dx_\cJ, \qquad x,\xi\in\rd.
	\]
	Since $\cF_\cJ\cF_\cK=\cF_{\cJ\cup\cK}$ for every $\cJ,\cK\subseteq\{1,\ldots,d\}$ disjoint, a direct consequence of \cite[Example 2.4]{cordero2023symplectic} shows that $\pi^{Mp}(\cF_\cJ)=\Pi_\cJ$, where $\Pi_\cJ$ is defined as in \eqref{defPiJ}. Observe that $\cF_\emptyset=Id_{L^2}$ and $\cF_{\{1,\ldots,d\}}=\cF$.\\
	(iii) For $L\in\mbox{GL}(d,\bR)$, the \emph{rescaling operator} $\mathfrak{T}_Lf=|\det(L)|^{1/2}f(L\cdot)$ is metaplectic with projection $\pi^{Mp}(\mathfrak{T}_L)=\cD_L$, as defined in \eqref{defVCDL}.\\
	(iv) Let $Q\in\mbox{Sym}(d,\bR)$ and consider the \emph{chirp function}:
	\begin{equation}\label{defPhi}
	\Phi_Q(x)=e^{i\pi Qx\cdot x}, \qquad x\in\rd.
	\end{equation}
	The \emph{chirp product} $\mathfrak{p}_Q f=\Phi_Qf$ is metaplectic with $\pi^{Mp}(\mathfrak{p}_Q)=V_Q$, defined as in \eqref{defVCDL}.\\
	(v) Analogously, for $P\in\Sym(d,\bR)$, the \emph{multipliers} $\mathfrak{m}_P f=\cF^{-1}\Phi_{-P}\ast f$ are metaplectic, with projections $\pi^{Mp}(\mathfrak{m}_P)=V_P^T$.
\end{example}

\begin{remark}
	Concerning the metaplectic multipliers defined in Example \ref{esMetap} (v), a straightforward computation, see e.g. \cite[Lemma 4.5]{cordero2024unified}, shows that if $P=\Sigma^T\Delta\Sigma$ with $\Delta=\mbox{diag}(\mbox{eig}(P))=\mbox{diag}(\lambda_1,\ldots,\lambda_d)$, and $\Sigma$ is the corresponding orthogonal matrix diagonalizing $C$, then:
	\begin{equation}\label{Fchirp}
	\cF^{-1}\Phi_{-P}=\gamma_P\mathfrak{T}_{\Sigma}\Big(\bigotimes_{j=1}^d\psi_j\Big),
	\end{equation}
	where $\gamma_P$ is a suitable constant, which depends on the non-zero eigenvalues of $P$,
	\begin{equation}\label{defpsi}
		\psi_j=\begin{cases}
			\delta_0 & \text{if $\lambda_j=0$},\\
			e^{i\pi(\cdot)^2/\lambda_j} & \text{if $\lambda_j\neq0$},
		\end{cases}
	\end{equation}
	and $\mathfrak{T}_\Sigma$ is defined as in Example \ref{esMetap} $(iii)$.
\end{remark}

In view of Proposition \ref{parametrization} and the fact that $\pi^{Mp}$ is a homomorphism, the examples above provide the building blocks to construct metaplectic operators.

\begin{proposition}\label{decompMp}
	Let $\hat S\in \Mp(d,\bR)$. There exist (non unique) $P,Q\in\Sym(d,\bR)$, $L\in\mbox{GL}(d,\bR)$ and $\cJ\subseteq\{1,\ldots,d\}$ such that:
	\begin{equation}\label{decompMpeq}
		\hat S =\mathfrak{p}_Q\mathfrak{T}_L\mathfrak{m}_P\cF_\cJ,
	\end{equation}
	up to a sign. 
\end{proposition}

\begin{definition}
	Let $\hat S\in \Mp(d,\bR)$. We refer to any factorization \eqref{decompMpeq} as to a \emph{Dopico-Johnson factorization} of $\hat S$.
\end{definition}

	Obviously, the boundedness properties of a metaplectic operator cannot depend on the Dopico-Johnson factorization chosen to decompose it. 

	

\subsection{Modulation spaces}
	The Wigner distribution defined in \eqref{intro5} for $L^2$ functions can be extended to $f,g\in\cS'(\rd)$ by means of metaplectic operators. Indeed, if $\cF_2=\cF_{\{d+1,\ldots,2d\}}\in\Mp(2d,\bR)$ is the partial Fourier transform with respect to the \emph{frequency variables} and
	\[
		L_{1/2}=\begin{pmatrix}
		I & I/2\\
		I & -I/2
		\end{pmatrix},
	\]
	then
	\begin{equation}\label{WonSp}
		W(f,g)=\cF_2\mathfrak{T}_{L_{1/2}}(f\otimes\bar g), \qquad f,g\in L^2(\rd).
	\end{equation}
Since $f\otimes\bar g$ is defined for every $f,g\in \cS'(\rd)$, \eqref{WonSp} extends to $f,g\in\cS'(\rd)$. If
\[
	L_{st}=\begin{pmatrix}
	0 & I\\
	-I & I\end{pmatrix},
\]
the short-time Fourier transform is defined as:
\begin{equation}\label{defMpnorm}
	V_gf=\cF_2\mathfrak{T}_{L_{st}}(f\otimes\bar g), \qquad f,g\in\cS'(\rd).
\end{equation}
If $f,g\in L^2(\rd)$, 
\[
	V_gf(x,\xi)=\int_{\rd}f(t)\overline{g(t-x)}e^{-2\pi i\xi t}dt, \qquad x,\xi\in\rd.
\]
If $f\in\cS'(\rd)$ and $g\in\cS(\rd)$, $W(f,g)$ and $V_gf$ define uniformly continuous functions on $\rdd$. For $0<p,q\leq\infty$ and $g\in\cS(\rd)\setminus\{0\}$ fixed, the quantities:
\begin{equation}\label{defMpnorm}
	\norm{f}_{M^{p,q}}=\norm{V_gf}_{L^{p,q}}, \qquad f\in\cS'(\rd)
\end{equation}
define (quasi-)norms on the subspaces of $\cS'(\rd)$, $$M^{p,q}(\rd)=\{f\in\cS'(\rd):\norm{f}_{M^{p,q}}<\infty\},$$ which are called \emph{modulation spaces}. Here, if $F:\rdd\to\bC$ is measurable,
\[
	\norm{F}_{L^{p,q}}=\norm{y\mapsto \norm{F(\cdot,y)}_p}_q.
\]

These spaces were defined by H. G. Feichtinger in the Banach case ($1\leq p,q\leq\infty$) in \cite{feichtinger1983modulation}, and later extended to the quasi-Banach setting ($0<p,q\leq\infty$) by Y. V. Galperin and S. Samarah in \cite{galperin2004time}.

Different $g\in\cS(\rd)\setminus\{0\}$ yield to equivalent (quasi-)norms. If $p=q$ we write $M^p(\rd)=M^{p,p}(\rd)$. Moreover, it is observed in \cite{de2011symplectic} that the Wigner distribution can be used to replace the short-time Fourier transform in \eqref{defMpnorm}:
\begin{equation}\label{MpWnorms}
	\norm{f}_{M^{p,q}}=\norm{W(f,g)}_{L^{p,q}}, \qquad f\in\cS'(\rd).
\end{equation}
The following relation between the short-time Fourier transform and the Wigner distribution is classical:
\begin{equation}\label{relSTFTW}
		W(f,g)(x,\xi)=2^de^{4\pi ix\xi}V_{\cI g}f(2x,2\xi),
\end{equation}
	for $f,g\in\ L^2(\rd)$, $x,\xi\in\rd$, and $\cI g(t)=g(-t)$ is the \emph{flip operator}.
	
In what follows, we will primarily use the Wigner distribution. The short-time Fourier transform will be employed mainly in the final section, which is more time-frequency analysis oriented, as it simplifies certain computations.

\section{The one-dimensional case}\label{sec:oneD}
	We reserve one section to report on the trivial case $d=1$, and in the next sections we will assume $d>1$. Let $S\in\Sp(1,\bR)=\{S\in\GL(2,\bR): \det(S)=1\}$,
	\[
		S=\begin{pmatrix}
			a & b\\
			c & d
		\end{pmatrix}, \qquad a,b,c,d\in\bR, \ ad-bc=1.
	\]
	
	We divide two cases.\\
	{\bf Case $b=0$.} If $b=0$, since $\det(S)=ad=1$, it must be:
	\[
		S=\begin{pmatrix}
			a & 0\\
			c & a^{-1}
		\end{pmatrix}.
	\]
	An easy computation shows that:
	\[
		S=\begin{pmatrix} 
			1 & 0\\
			ca^{-1} & 1
		\end{pmatrix}
		\begin{pmatrix}
			a & 0\\
			0 & a^{-1}
		\end{pmatrix},
	\]
	so that the associated metaplectic operator $\hat S$ has Dopico-Johnson factorization $\hat S=\mathfrak{p}_{ca^{-1}}\mathfrak{T}_{a^{-1}}$ (see Examples \ref{esMetap} for the definitions of these operators), i.e., $\hat Sf(t)=|a|^{-1/2}e^{i\pi ca^{-1}t^2}f(t/a)$ up to a sign, which is a surjective quasi-isometry of $L^p(\bR)$ for every $0<p\leq\infty$.\\
\medskip
	{\bf Case $b\neq0$.} This is the case of free symplectic matrices. An easy computation shows that in this case $S$ has Dopico-Johnson factorization:
	\[
		S=\begin{pmatrix}
			1 & 0\\
			db^{-1} & 0
		\end{pmatrix}
		\begin{pmatrix}
			b & 0\\
			0 & b^{-1}
		\end{pmatrix}
		\begin{pmatrix}
		0 & 1\\
		-1 & 0
		\end{pmatrix}
		\begin{pmatrix}
			1 & 0\\
			b^{-1}a & 1
		\end{pmatrix},
	\]
	or, equivalently, the associated metaplectic operator is $\hat S=\mathfrak{p}_{db^{-1}}\mathfrak{T}_{b^{-1}}\cF\mathfrak{p}_{b^{-1}a}$ (see Example \ref{esMetap} for the definitions of these operators), which is bounded from $L^p(\bR)$ to $L^q(\bR)$ if and only if $1\leq p\leq2$ and $q=p'$, due to the presence of the Fourier transform.\\
	
	We conclude this section with a synthesis of the boundedness properties of metaplectic operators on $L^p$ spaces in the straightforward univariate case.
	
	\begin{theorem}
		Let $\hat S\in \Mp(d,\bR)$ and $S=\pi^{Mp}(d,\bR)$. Then,\\
		(i) If $b=0$, $\hat S$ is a surjective quasi-isometry of $L^p(\bR)$ for every $0<p\leq\infty$, with $\norm{\hat S}_{B(L^p)}=|a|^{1/2-1/p}$.\\
		(ii) If $b\neq0$, then $\hat S\in B(L^p(\rd),L^q(\rd))$ if and only if $1\leq p\leq 2$ and $q=p'$. 
	\end{theorem}
	
	Surprisingly, the casuistry in the multivariate case is not wider.

\section{Dopico-Johnson factorizations of (non-)free metaplectic operators}\label{sec:free}

Let $S\in \Sp(d,\bR)$ be a free symplectic matrix. This means that the upper-right block of $S$ in its block decomposition \eqref{blockS} is invertible. The classical factorization of the symplectic projection of a free metaplectic operator can be presented in two forms:
\begin{align}\label{factFree}
	S&=V_{DB^{-1}}\cD_{B^{-1}}JV_{-B^{-1}A}\\
	\label{factFree2}
	&=V_{DB^{-1}}\cD_{B^{-1}}V^T_{AB^{-T}}J,
\end{align}
see \eqref{defVCDL} for the definitions of the matrices appearing in \eqref{factFree} and \eqref{factFree2}. Formula \eqref{factFree}, from the metaplectic operators perspective, reads as $\hat S=\mathfrak{p}_{DB^{-1}}$$\mathfrak{T}_{B^{-1}}$$\cF\mathfrak{p}_{-B^{-1}A}$ up to a sign. Hence, $\hat S$ inherits the boundedness properties of the Fourier transform, i.e. it is bounded from $L^p(\rd)$ to $L^q(\rd)$ if and only if $1\leq p\leq 2$ and $q=p'$. 

\begin{theorem}
	Let $\hat S\in \Mp(d,\bR)$ be a free metaplectic operator. Then, $\hat S\in B(L^p(\rd),L^q(\rd))$ if and only if $1\leq p\leq 2$ and $q=p'$. In those cases,
	\[
		\norm{\hat S}_{B(L^p,L^{p'})}=|\det(B)|^{-1/2+1/p}\Big(\frac{p^{1/p}}{(p')^{1/p'}}\Big)^{d/2}.
	\]
\end{theorem}

In this section, we characterize the Dopico-Johnson factorizations of free symplectic matrices, which yields to a deeper understanding of the complementary case in which $\det(B)=0$. Let $S=V_Q\cD_LV_P^T\Pi_\cJ$ any Dopico-Johnson factorization of $S$, where the matrices appearing in the product are defined as in \eqref{defVCDL} and \eqref{defPiJ}. The property of being free cannot depend on the choice of the Dopico-Johnson factorization, which for fixed $P,Q\in\Sym(d,\bR)$, $L\in\GL(d,\bR)$ and $\cJ\subseteq\{1,\ldots,d\}$ can be computed explicitly:
\begin{align}
		\label{partial-0}
			S&=\begin{pmatrix}
				I & 0\\
				Q & I
			\end{pmatrix}
			\begin{pmatrix}
				L^{-1} & 0\\
				0 & L^T
			\end{pmatrix}
			\begin{pmatrix}
				I & P\\
				0 & I
			\end{pmatrix}
			\begin{pmatrix}
				I_{\cJ^c} & I_{\cJ}\\
				-I_{\cJ} & I_{\cJ^c}
			\end{pmatrix}\\
			\label{partial-1}
			&=\begin{pmatrix}
				L^{-1} & 0\\
				QL^{-1} & L^T
			\end{pmatrix}
			\begin{pmatrix}
				I_{\cJ^c}-PI_{\cJ} & I_\cJ+PI_{\cJ^c}\\
				-I_\cJ & I_{\cJ^c}
			\end{pmatrix}\\
			\label{eqFree}
			&=\begin{pmatrix}
				L^{-1}(I_{\cJ^c}-PI_{\cJ}) & L^{-1}(I_\cJ+PI_{\cJ^c})\\
				QL^{-1}(I_{\cJ^c}-PI_{\cJ})-L^TI_{\cJ} & QL^{-1}(I_\cJ+PI_{\cJ^c})+L^TI_{\cJ^c}
			\end{pmatrix}.
		\end{align}
		The matrix \eqref{eqFree} is free if and only if $L^{-1}(I_\cJ+PI_{\cJ^c})\in \GL(d,\bR)$ or, equivalently, $I_\cJ+PI_{\cJ^c}\in\GL(d,\bR)$. We simplify this latter condition further in the following lemma.
		
		\begin{lemma}\label{lemmaFree}
			Let $P=(p_{jk})_{j,k=1}^d\in\Sym(d,\bR)$ and $\cJ\subsetneqq\{1,\ldots,d\}$. Then, $I_\cJ+PI_{\cJ^c}\in\GL(d,\bR)$ if and only if the submatrix $P_{\cJ^c\cJ^c}=(p_{jk})_{j,k\in\cJ^c}$ is invertible.
		\end{lemma}
		\begin{proof}
			If $\cJ=\emptyset$, then $I_\cJ+PI_{\cJ^c}=P_{\cJ^c\cJ^c}=P$ and the equivalence follows trivially. Assume $\cJ\subsetneqq\{1,\ldots,d\}$ and $\cJ\neq\emptyset$, and let $1\leq r<d$ be the cardinality of $\cJ$. By the expression \eqref{eqFree} and the invertibility of $L$, $S$ is free if and only if $I_\cJ+PI_{\cJ^c}\in\GL(d,\bR)$. We need to show that $I_\cJ+PI_{\cJ^c}\in\GL(d,\bR)$ if and only if $P_{\cJ^c\cJ^c}\in\GL(d,\bR)$. Let $W_L,W_R\in\GL(d,\bR)$ be the permutation matrices such that if $A\in\bR^{d\times d}$,
	\[
		W_LAW_R=\left(\begin{array}{c|c}
			A_{\cJ\cJ} & A_{\cJ\cJ^c}\\
			\hline
			A_{\cJ^c\cJ} & A_{\cJ^c\cJ^c}
		\end{array}\right).
	\]
	Obviously, the matrix $I_\cJ+PI_{\cJ^c}$ is invertible if and only if $W_L(I_\cJ+PI_{\cJ^c})W_R$ is invertible. Since $(PI_{\cJ^c})_{\cJ\cJ^c}$ selects the columns of $PI_{\cJ^c}$ indexed by $\cJ^c$, we have $(PI_{\cJ^c})_{\cJ\cJ^c}=P_{\cJ\cJ^c}$, and analogously $(PI_{\cJ^c})_{\cJ^c\cJ^c}=P_{\cJ^c\cJ^c}$. Therefore,
	\begin{align*}
		W_L(I_\cJ+PI_{\cJ^c})W_R&=W_LI_\cJ W_R+W_LPI_{\cJ^c}W_R\\
		&=\left(\begin{array}{c|c}
			I_{r\times r} &0_{r\times(d-r)} \\
			\hline
			0_{(d-r)\times r} & 0_{(d-r)\times(d-r)} \\
		\end{array}
		\right)+
		\left(\begin{array}{c|c}
			0_{r\times r} & (PI_{\cJ^c})_{\cJ\cJ^c}\\ 
			\hline
			0_{(d-r)\times r} & (PI_{\cJ^c})_{\cJ^c\cJ^c} \\
		\end{array}
		\right)\\
		&=\left(\begin{array}{c|c}
			I_{r\times r} &0_{r\times(d-r)} \\
			\hline
			0_{(d-r)\times r} & 0_{(d-r)\times(d-r)} \\
		\end{array}
		\right)+
		\left(\begin{array}{c|c}
			0_{r\times r} & P_{\cJ\cJ^c}\\ 
			\hline
			0_{(d-r)\times r} & P_{\cJ^c\cJ^c} \\
		\end{array}
		\right)\\
		&=\left(\begin{array}{c|c}
			I_{r\times r} & P_{\cJ\cJ^c} \\
			\hline
			0_{(d-r)\times r} & P_{\cJ^c\cJ^c} \\
		\end{array}
		\right).
	\end{align*}
	By Schur's formula, 
	\[
		|\det(W_L(I_\cJ+PI_{\cJ^c})W_R)|=|\det(P_{\cJ^c\cJ^c})|,
	\]
	from which it follows that the matrix $W_L(I_\cJ+PI_{\cJ^c})W_R$ is invertible if and only if $P_{\cJ^c\cJ^c}\in\GL(d,\bR)$. This concludes the proof.

		\end{proof}
	
\begin{corollary}\label{corFree}
	Let $S\in\Sp(d,\bR)$. The following statements are equivalent.\\
	(i) $S$ is free.\\
	(ii) Every Dopico-Johnson factorizations of $S$ satisfy one of the following properties:\\
	\qquad(ii.1) $\cJ=\{1,\ldots,d\}$.\\
	\qquad(ii.2) $\cJ\subsetneqq\{1,\ldots,d\}$ and $P_{\cJ^c\cJ^c}\in\GL(d,\bR)$.
\end{corollary}
\begin{proof}
	The implication $(i)\Rightarrow(ii)$ follows trivially by the Dopico-Johnson factorization in \eqref{factFree}, which has $\cJ=\{1,\ldots,d\}$. The converse is straighforward: assume that there exists a Dopico-Johnson factorization $S=V_Q\cD_LV^T_P\Pi_\cJ$ with $\cJ\subsetneqq\{1,\ldots,d\}$ and $P_{\cJ^c\cJ^c}$ singular. Then, $L^{-1}(I_\cJ+PI_{\cJ^c})\notin\GL(d,\bR)$ by Lemma \ref{lemmaFree}. Consequently, $S$ is not free. If, instead, $\cJ=\{1,\ldots,d\}$, then $B=L^{-1}\in\GL(d,\bR)$. This concludes the proof.
	
\end{proof}

\begin{corollary}\label{corolla44}
	Let $\hat S\in\Mp(d,\bR)$. Then, one of the following statements hold, then either: \\
	(i) $\hat S$ is free, or \\
	(ii) $\hat S$ can be factorized as $\hat S=\mathfrak{p}_Q\mathfrak{T}_L\mathfrak{m}_P\cF_\cJ$, for some $\cJ\subsetneqq\{1,\ldots,d\}$ and $P\in\Sym(d,\bR)$ such that $P_{\cJ^c\cJ^c}\notin GL(d,\bR)$.
\end{corollary}

\section{Homeomorphisms of $L^p(\rd)$}\label{sec:B0}
In view of Theorem \ref{thmintro1}, to conclude the characterization of $L^p$ boundedness of metaplectic operators, it remains to understand the boundedness of non-free metaplectic operators, i.e. metaplectic operators that factorize as $\hat S=\mathfrak{p}_Q\mathfrak{T}_L\mathfrak{m}_P\cF_\cJ$ for some $P,Q\in\Sym(d,\bR)$, $L\in\GL(d,\bR)$ and $\cJ\subsetneqq\{1,\ldots,d\}$ such that $P_{\cJ^c\cJ^c}$ is singular. 

If $\hat S=\mathfrak{p}_Q\mathfrak{T}_L$, $\hat S$ is trivially a homeomorphism of $L^p(\rd)$ for every $0<p\leq\infty$. The core of this section is proving the converse, i.e., if $\hat S$ is a homeomorphism of $L^p(\rd)$, then its only possible Dopico-Johnson factorization can be $\hat S=\mathfrak{p}_Q\mathfrak{T}_L$. Equivalently, from the symplectic group perspective, $\hat S$ is a homeomorphism of $L^p(\rd)$ for every $0<p\leq\infty$ if and only if its projection $S$ has block decomposition:
\begin{equation}\label{blockS2}
	S=\begin{pmatrix}
		A & 0\\
		C & D
	\end{pmatrix}, \qquad A,C,D\in\bR^{d\times d}.
\end{equation}

For the benefit of the presentation, we begin emphasizing some well-known fact about symplectic matrices which decompose as in \eqref{blockS2}, and the related metaplectic operators.


	\begin{remark}\label{rem-f1}
		Let $S\in \Sp(d,\bR)$ have block decomposition \eqref{blockS2}. The symplectic relations \eqref{defSp} for $S$ can be rephrased as follows:
		\[
			\begin{cases}
			CA^{-1}\in\Sym(d,\bR),\\
			D=A^{-T}.
			\end{cases}
		\]
		Moreover, a direct computation shows that $S$ can be factorized as:
		\begin{equation}\label{easyprod0}
			S=V_{CA^{-1}}\cD_{A^{-1}},
		\end{equation}
		or
		\begin{equation}\label{easyprod3}
			S=\cD_{A^{-1}}V_{A^TC},
		\end{equation}
		where the matrices at the right hand-sides are defined as in \eqref{defVCDL}. From the metaplectic operators perspective, 
		\begin{equation}\label{easyprod1}
			\hat S=\mathfrak{p}_{CA^{-1}}\mathfrak{T}_{A^{-1}},
		\end{equation}
		or 
		\begin{equation}\label{easyprod2}
			\hat S=\mathfrak{T}_{A^{-1}}\mathfrak{p}_{A^TC}, 
		\end{equation}
		up to a sign. As observed in Remark \ref{rem-par}, in this work we consider the factorization \eqref{easyprod0}. This choice is irrelevant: all of our findings can be rephrased by reverting the order of the two operators. 
	\end{remark}
	
	\begin{remark}\label{rem-34}
		Let $S\in\Sp(d,\bR)$ have block representation \eqref{blockS2}. The factorization \eqref{easyprod0} is the unique representation of $S$ as a product $V_Q\cD_L$, $Q\in \Sym(d,\bR)$, $L\in\GL(d,\bR)$. Indeed, if $S=V_Q\cD_L=V_{Q'}\cD_{L'}$, using \eqref{defVCDL}:
		\[
			V_Q\cD_L=V_{Q'}\cD_{L'}\Leftrightarrow\begin{pmatrix}
				L^{-1} & 0\\
				QL^{-1} & L^T
			\end{pmatrix}=\begin{pmatrix}
				{L'}^{-1} & 0\\
				Q'{L'}^{-1} & {L'}^T
			\end{pmatrix},
		\]
		from which the uniqueness follows.
	\end{remark}
	\begin{remark}
		Let $S\in\Sp(d,\bR)$ have block representation \eqref{blockS2}. Formula \eqref{easyprod0} is also a Dopico-Johnson factorization of $S$. If we prove that any Dopico-Johnson factorization of $S$ has $P=0$ and $\cJ=\emptyset$, Remark \eqref{rem-34} would entail that $S$ admits a unique Dopico-Johnson factorization, namely $S=V_{CA^{-1}}\cD_{A^{-1}}$, or equivalently, that \eqref{easyprod1} is the unique Dopico-Johnson factorization of $\hat S$, up to a sign.
	\end{remark}
	
	The factorization \eqref{easyprod0} is actually a characterization of symplectic matrices $S\in\Sp(d,\bR)$ having lower block triangular representation \eqref{blockS2}. Indeed, if $S=V_Q\cD_L$ for some $Q\in\Sym(d,\bR)$ and $L\in\GL(d,\bR)$, then
		\[
			S=\begin{pmatrix}
				L^{-1} & 0\\
				QL^{-1} & L^T
			\end{pmatrix}
		\]
		is in the form \eqref{blockS2}. In the following lemma, we prove that if $S$ has block decomposition \eqref{blockS2}, then there cannot be $P\in\Sym(d,\bR)\setminus\{0\}$ and $\emptyset\neq\cJ\subseteq\{1,\ldots,d\}$ so that $V_P^T\Pi_\cJ=I$. 

	\begin{lemma}\label{lemma1}
		Let $S\in\Sp(d,\bR)$ have block decomposition \eqref{blockS}. The following are equivalent.\\
		(i) $B=0$.\\
		(ii) Every Dopico-Johnson factorization $S=V_Q\cD_LV_P^T\Pi_\cJ$ has $P=0$ and $\cJ=\emptyset$.\\
		(iii) $\hat S$ factorizes (uniquely) as $\hat S=\mathfrak{p}_Q\mathfrak{T}_L$, for some $Q\in\Sym(d,\bR)$ and $L\in\GL(d,\bR)$.
	\end{lemma}
	\begin{proof}
		The projection $\pi^{Mp}$ is a homomorphism with $\ker(\pi^{Mp})=\{\pm Id_{L^2}\}$, whence the equivalence $(ii)\Leftrightarrow(iii)$, the uniqueness in $(iii)$ following by Remark \ref{rem-34}.
		
		Let us show the equivalence $(i)\Leftrightarrow(ii)$. It is obvious that if every Dopico-Johnson factorization of $S$ has $\cJ=\emptyset$ and $P=0$, then $B=0$. We prove the converse. Assume that $S=V_Q\cD_LV_P^T\Pi_\cJ$ with $\cJ\neq\emptyset$. Clearly, $B=L^{-1}(I_\cJ+PI_{\cJ^c})=0$ if and only if $I_\cJ+PI_{\cJ^c}=0$. This matrix can be 0 if and only if $\cJ=\emptyset$ and $P=0$, and this is straightforward. Let us show the contradiction more precisely. In view of $I_\cJ+PI_{\cJ^c}=0$, \eqref{partial-1} reads as:
		\begin{equation}\label{factProof1}
			S=\begin{pmatrix}
				L^{-1} & 0\\
				QL^{-1} & L^T
			\end{pmatrix}
			\begin{pmatrix}
				I_{\cJ^c}-PI_{\cJ} & 0\\
				-I_\cJ & I_{\cJ^c}
			\end{pmatrix}.
		\end{equation}
		The matrix 
		\[
			\begin{pmatrix}
				I_{\cJ^c}-PI_{\cJ} & 0\\
				-I_\cJ & I_{\cJ^c}
			\end{pmatrix}
		\]
		shall be symplectic, however $I_{\cJ^c}\notin\GL(d,\bR)$, so that it cannot even be invertible. This proves that it must be $\cJ=\emptyset$. 
		
		Next, we use this information to show that also $P=0$. The identity \eqref{partial-0}, with $\cJ=\emptyset$, reads as:
		\begin{equation}\label{MMult}
			S=\begin{pmatrix}
				I & 0\\
				Q & I
			\end{pmatrix}
			\begin{pmatrix}
				L^{-1} & 0\\
				0 & L^T
			\end{pmatrix}
			\begin{pmatrix}
				I & P\\
				0 & I
			\end{pmatrix}=\begin{pmatrix}
				L^{-1} & L^{-1}P\\
				QL^{-1} & QL^{-1}P+L^T
			\end{pmatrix}.
		\end{equation}
		Since $L\in\GL(d,\bR)$, $B=L^{-1}P=0$ if and only if $P=0$. This completes the proof.
		
	\end{proof}
%

	The following Lemma is the key tool we use in this work to find counterexamples.
	
	\begin{lemma}\label{lemma-redox}
		Let $\cJ\subseteq\{1,\ldots,d\}$ and $P\in\Sym(d,\bR)$. Then,
		\begin{equation}\label{new-decomp}
			\Pi_\cJ^{-1} V_P^T\Pi_\cJ=V_{I_{\cJ^c}PI_{\cJ^c}}^T\cD_{I+I_{\cJ^c} P I_\cJ}V_{-I_\cJ P I_{\cJ}}.
		\end{equation}
		Consequently, up to a sign,
		\begin{equation}\label{new-decomp-met}
			\mathfrak{m}_P\cF_\cJ=\cF_\cJ\mathfrak{m}_{I_{\cJ^c}PI_{\cJ^c}}(\mathfrak{T}_{I+I_{\cJ^c} P I_\cJ}\mathfrak{p}_{-I_\cJ P I_{\cJ}}).
		\end{equation}
	\end{lemma}
	\begin{proof}
		Using that: $I_{\cJ}I_{\cJ^c}=I_{\cJ^c}I_\cJ=0$, $I_\cJ^2=I_\cJ$, $I_{\cJ^c}^2=I_{\cJ^c}$ and $I_{\cJ}+I_{\cJ^c}=I$, and observing that $I_{\cJ^c}PI_{\cJ^c},I_{\cJ}PI_{\cJ}\in\Sym(d,\bR)$ and $I-I_{\cJ^c}PI_{\cJ}\in GL(d,\bR)$ with $(I-I_{\cJ^c}PI_{\cJ})^{-1}=I+I_{\cJ^c}PI_\cJ$, we have:
		\begin{align*}
			\Pi_\cJ^{-1} V_P^T\Pi_\cJ&=\begin{pmatrix}
				I_{\cJ^c} & -I_{\cJ}\\
				I_{\cJ} & I_{\cJ^c}
			\end{pmatrix}\begin{pmatrix}
				I & P\\
				0 & I
			\end{pmatrix}\begin{pmatrix}
				I_{\cJ^c} & I_{\cJ}\\
				-I_{\cJ} & I_{\cJ^c}
			\end{pmatrix}\\
			&=\begin{pmatrix}
				I_{\cJ^c} & I_{\cJ^c} P-I_{\cJ}\\
				I_{\cJ} & I_{\cJ}P+I_{\cJ^c}
			\end{pmatrix}\begin{pmatrix}
				I_{\cJ^c} & I_{\cJ}\\
				-I_{\cJ} & I_{\cJ^c}
			\end{pmatrix}\\
			&=\begin{pmatrix}
				I_{\cJ^c}-I_{\cJ^c}PI_{\cJ}+I_{\cJ} & I_{\cJ^c}PI_{\cJ^c}\\
				-I_{\cJ}PI_{\cJ} & I_\cJ+I_{\cJ}PI_{\cJ^c}+I_{\cJ^c}
			\end{pmatrix}\\
			&=\begin{pmatrix}
				I-I_{\cJ^c}PI_{\cJ} & I_{\cJ^c}PI_{\cJ^c}\\
				-I_{\cJ}PI_{\cJ} & I+I_{\cJ}PI_{\cJ^c}
			\end{pmatrix}\\
			&=\begin{pmatrix}
				I & I_{\cJ^c}PI_{\cJ^c}\\
				0 & I
			\end{pmatrix}\begin{pmatrix}
				I-I_{\cJ^c}PI_{\cJ} & 0\\
				0 & I+I_{\cJ}PI_{\cJ^c}
			\end{pmatrix}\begin{pmatrix}
				I & 0\\
				-I_{\cJ}PI_{\cJ} & I
			\end{pmatrix}\\
			&=\begin{pmatrix}
				I & I_{\cJ^c}PI_{\cJ^c}\\
				0 & I
			\end{pmatrix}\begin{pmatrix}
				(I+I_{\cJ^c}PI_{\cJ})^{-1} & 0\\
				0 & (I+I_{\cJ^c}PI_{\cJ})^T
			\end{pmatrix}\begin{pmatrix}
				I & 0\\
				-I_{\cJ}PI_{\cJ} & I
			\end{pmatrix},
		\end{align*}
		and the assertion follows.
	\end{proof}
	
	Lemma \ref{lemma-redox} entails the following reduction result.
	
	\begin{corollary}\label{cor-deco}
		Let $\hat S\in\Mp(d,\bR)$, and $0<p,q\leq\infty$. Let $\hat S=\mathfrak{p}_Q\mathfrak{T}_L\mathfrak{m}_P\cF_\cJ$ be a Dopico-Johnson factorization of $\hat S$. The following are equivalent:\\
		(i) $\hat S$ is bounded from $L^p(\rd)$ to $L^q(\rd)$.\\
		(ii) $\mathfrak{m}_P\cF_\cJ$ is bounded from $L^p(\rd)$ to $L^q(\rd)$.\\
		(iii) $\cF_\cJ\mathfrak{m}_{I_{\cJ^c}PI_{\cJ^c}}$ is bounded from $L^p(\rd)$ to $L^q(\rd)$.
	\end{corollary}
	\begin{proof}
		The equivalence $(i)\Leftrightarrow(ii)$ is trivial. Similarly, the equivalence $(ii)\Leftrightarrow(iii)$ follows from \eqref{new-decomp-met}, since $\mathfrak{T}_{I+I_{\cJ^c} P I_\cJ}\mathfrak{p}_{-I_\cJ P I_{\cJ}}$ is a homeomorphism of $L^p(\rd)$.
	\end{proof}
		
	\begin{theorem}\label{main-thm-ref}
		Let $\hat S\in\Mp(d,\bR)$ have projection $S=\pi^{Mp}(d,\bR)$, with block decomposition \eqref{blockS}. The following statements are equivalent.\\
		(i) $B=0$.\\
		(ii) $\hat S:L^p(\rd)\to L^p(\rd)$ is bounded for every $0<p\leq\infty$.\\
		(iii) $\hat S$ is a surjective quasi-isometry of $L^p(\rd)$ for every $0<p\leq\infty$, with $\norm{\hat S}_{B(L^p)}=|\det(A)|^{1/2-1/p}$.\\
		(iv) $\hat S=\mathfrak{p}_Q\mathfrak{T}_L$ for some (unique) $Q\in \Sym(d,\bR)$ and $L\in GL(d,\bR)$. Namely, $Q=CA^{-1}$ and $L=A^{-1}$.
	\end{theorem}
	\begin{proof}
		The implications $(i)\Rightarrow(ii),(iii),(iv)$ are proved in Lemma \ref{lemma1}. The implications $(iv)\Rightarrow(i),(ii),(iii)$ are trivial. The following diagram summarizes the implications that we have for free.
		
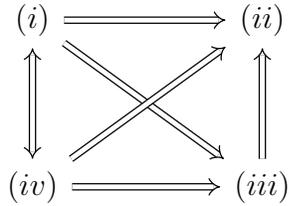
\begin{center}
		\begin{tikzcd}[row sep=huge, column sep=huge]
    			|[alias=ii]| (i) \arrow[r,Rightarrow]\arrow[dr,Rightarrow] & |[alias=i]| (ii) \\
   			|[alias=iii]| (iv) \arrow[u,Leftrightarrow]\arrow[ur,Rightarrow] \arrow[r,Rightarrow] & |[alias=iv]| (iii) \arrow[u,Rightarrow] 
\end{tikzcd}
		\end{center}
		To sum up, $(i)\Rightarrow(iv)\Rightarrow(iii)\Rightarrow(ii)$ follow by an elementary argument. To close the loop, we need to prove the implication $(ii)\Rightarrow(i)$. Specifically, we show that if $\hat S$ is bounded on $L^p(\rd)$ for some $p\neq2$, and $S=V_Q\cD_LV_P^T\Pi_\cJ$ is any Dopico-Johnson factorization of $S$, then it must be the case that $P=0$ and $\cJ=\emptyset$. According to Lemma \ref{lemma1}, this is equivalent to having $B=0$. By contradiction, assume that $\hat S$ is bounded on $L^p(\rd)$ for some $p\neq2$ and that there exist $P,Q\in\Sym(d,\bR)$, $L\in\GL(d,\bR)$ and $\cJ\subseteq\{1,\ldots,d\}$ such that the corresponding Dopico-Johnson factorization $\hat S=V_Q\cD_LV_P^T\cF_\cJ$ has either $P\neq0$ or $\cJ\neq\emptyset$. By Corollary \ref{cor-deco}, the contradiction follows if we prove that $\cF_\cJ\mathfrak{m}_{I_{\cJ^c}PI_{\cJ^c}}$ is unbounded on $L^p(\rd)$.
		
		The case $\cJ=\emptyset$ is classical and it is is due to L. H\"ormander \cite[Lemma 1.4]{hormander1960estimates}. 
		
		Analogously, the case $\cJ^c=\emptyset$ follows from:
		\[
			\cF_\cJ\mathfrak{m}_{I_{\cJ^c}PI_{\cJ^c}}=\cF,
		\]
		which is not bounded from $L^p(\rd)$ to itself for any $p\neq2$. 
		
		Let us consider the case $\cJ,\cJ^c\neq\emptyset$. 
		Consider $f(x)=g(x_\cJ)h(x_{\cJ^c})$, for $g\in L^p(\bR^r)$ and $h\in L^p(\R^{d-r})$ to be fixed, where $1\leq r<d$ is the cardinality of $\cJ$. 
		\begin{align}
		\nonumber
			\cF_\cJ\mathfrak{m}_{I_{\cJ^c}PI_{\cJ^c}}f(\xi)&=\cF_\cJ(\cF^{-1}\Phi_{-I_{\cJ^c}PI_{\cJ^c}}\ast f)(\xi)\\
			\nonumber
			&=\cF_\cJ\cF^{-1}(\Phi_{-I_{\cJ^c}PI_{\cJ^c}}\hat f)(\xi)\\
			\nonumber
			&=\cF_{\cJ^c}^{-1}(\Phi_{-I_{\cJ^c}PI_{\cJ^c}}\hat f)(\xi)\\
			\label{eqUnbound}
			&=\cF_{\cJ^c}^{-1}(\Phi_{-P_{\cJ^c\cJ^c}}\cF_{\cJ^c} h)(\xi_{\cJ^c}) \cF_\cJ g(\xi_\cJ).
		\end{align}
		Since $g$ depends only on the variables indexed by $\cJ$, we use (now and in the following) the simplified notation: $\cF_\cJ g(\xi_\cJ)=\hat g(\xi_\cJ)$. Similarly,
		\[
			\cF_{\cJ^c}^{-1}(\Phi_{-P_{\cJ^c\cJ^c}}\cF_{\cJ^c} h)(\xi_{\cJ^c})=\cF^{-1}(\Phi_{-P_{\cJ^c\cJ^c}}\hat h)(\xi_{\cJ^c}).
		\] 
		The contradiction follows by choosing $g\in L^p(\bR^r)$ so that $\hat g\notin L^p(\bR^r)$. This proves the implication $(ii)\Rightarrow(i)$, and we are done.
		
	\end{proof}

\section{Unbounded metaplectic operators}\label{sec:unbound}

We observed that if $\hat S$ is free, then $\hat S\in B(L^p(\rd),L^q(\rd))$ if and only if $1\leq p\leq2$ and $q=p'$, and we characterized metaplectic homeomorphisms of $L^p(\rd)$. In view of Corollary \ref{corolla44}, it remains to study the case in which $\hat S$ has Dopico-Johnson factorization $\hat S=\mathfrak{p}_Q\mathfrak{T}_L\mathfrak{m}_P\cF_\cJ$ for some $P,Q\in \Sym(d,\bR)$, $L\in\GL(d,\bR)$ and $\cJ\subsetneqq\{1,\ldots,d\}$, with $P_{\cJ^c\cJ^c}$ singular.

\begin{theorem}
	Let $\hat S\in\Mp(d,\bR)$ be a non-free metaplectic operator. Then, for every $0<p,q\leq\infty$, $\hat S\notin B(L^p(\rd),L^q(\rd))$. 
\end{theorem}
\begin{proof}
	By Corollary \ref{corolla44}, $\hat S$ has a Dopico-Johnson factorization $\hat S=\mathfrak{p}_Q\mathfrak{T}_L\mathfrak{m}_P\cF_\cJ$ for some $P,Q\in\Sym(d,\bR)$, $L\in\GL(d,\bR)$ and $\cJ\subsetneqq\{1,\ldots,d\}$, with $P_{\cJ^c\cJ^c}$ singular. 
	
	Let $0<p,q\leq\infty$. By Lemma \ref{lemma-redox}, it is clearly enough to prove that the operator $\cF_\cJ\mathfrak{m}_{I_{\cJ^c}PI_{\cJ^c}}$ is unbounded. By \eqref{eqUnbound}, choosing 
	\begin{equation}\label{deff}
		f(x)=g(x_\cJ)h(x_{\cJ^c}), \qquad x\in\rdd,
	\end{equation} 
	as in the proof of Theorem \ref{main-thm-ref}, we retrieve:
	\begin{equation}\label{bohEq}
		\cF_\cJ\mathfrak{m}_{I_{\cJ^c}PI_{\cJ^c}}f(\xi)=\cF^{-1}(\Phi_{-P_{\cJ^c\cJ^c}}\hat h)(\xi_{\cJ^c})\hat g(\xi_{\cJ}).
	\end{equation}
	First, we prove the assertion for $\cJ\neq\emptyset$. Let $0< r<d$ be the cardinality of $\cJ$. Consider the diagonalization $P_{\cJ^c\cJ^c}=\Sigma^T\Delta\Sigma$, where $\Sigma^T\Sigma=I$ and $\Delta=\mbox{diag}(\mbox{eig}(P_{\cJ^c\cJ^c}))=\{\lambda_1,\ldots,\lambda_{d-r}\}$. To simplify the notation, let us write $\eta=\xi_{\cJ^c}$. Interpreting the integrals in the distributional sense,
	\begin{align*}
		\cF^{-1}(\Phi_{-P_{\cJ^c\cJ^c}}\hat h)(\eta)&=\int_{\bR^{d-r}}\Phi_{-\Delta}(\Sigma x)\hat h(x)e^{2\pi i\eta x}dx\\
		&=\int_{\bR^{d-r}}\Phi_{-\Delta}(y)\hat h(\Sigma^Ty)e^{2\pi iy\cdot\Sigma\eta}dy.
	\end{align*}
	Since $h\in L^p(\bR^{d-r})$, the function $h':=\mathfrak{T}_\Sigma h$ belongs to $L^p(\bR^{d-r})$. Continuing the computation:
	 \begin{align*}
	 	\cF^{-1}(\Phi_{-P_{\cJ^c\cJ^c}}\hat h)(\eta)&=\int_{\R^{d-r}}\Phi_{-\Delta}(y)\widehat{h'}(y)e^{2\pi iy\cdot\Sigma\eta}dy\\
		&=\cF^{-1}(\Phi_{-\Delta}\widehat{h'})(\Sigma \eta),
	 \end{align*}
	 which belongs to $L^q(\bR^{d-r})$ if and only if $\cF^{-1}(\Phi_{-\Delta}\widehat{h'})\in L^q(\bR^{d-r})$. To conclude the proof it is enough to exhibit $h'\in L^p(\bR^{d-r})$ so that $\cF^{-1}(\Phi_{-\Delta}\widehat{h'})\notin L^q(\bR^{d-r})$. Choose $h'=\bigotimes_{k=1}^{d-r}h'_k$, with $h_1',\ldots,h_{{d-r}}'\in L^p(\bR)$. Let $\cK=\{k=1,\ldots,d-r : \lambda_k=0\}$, which is non-empty, since $P_{\cJ^c\cJ^c}$ is singular. Then, 
	 \begin{align*}
	 	\cF^{-1}(\Phi_{-\Delta}\widehat{h'})(\eta)&=\cF^{-1}\Phi_{-\Delta}\ast h'(\eta)\\
		&=\gamma_{-\Delta}\prod_{k\in\cK}\delta_0\ast h'_k(\eta_k)\prod_{k\notin\cK}e^{i\pi\lambda_k(\cdot)^2}\ast h'_k(\eta_k)\\
		&=\gamma_{-\Delta}\prod_{k\in\cK} h'_k(\eta_k)\prod_{k\notin\cK}e^{i\pi\lambda_k(\cdot)^2}\ast h'_k(\eta_k),
	 \end{align*}
	 for a suitable constant $\gamma_{-\Delta}\in\bC$. The assertion for the case $\cJ,\cJ^c\neq\emptyset$ follows fixing $k\in\cK$, and choosing the corresponding $h'_k\in L^p(\bR)\setminus L^q(\bR)$. Remarkably, the same rationale applies when $\cJ=\emptyset$. In this scenario, $g$ does not appear in \eqref{deff}, and thus is absent in \eqref{bohEq}. Since the previous argument involved selecting $h$, rather than  $g$, it equally establishes the claim for $\cJ=\emptyset$.
	 
\end{proof}

\section{Applications to pseudodifferential operators}\label{sec:Appl}
Our motivation for investigating the $L^p$ boundedness of metaplectic operators stems from the result highlighted in \cite[Theorem 3.8]{cordero2024unified}, which addresses the boundedness of pseudodifferential operators. In the following, we prioritize presenting this result over delving into the details of the objects appearing in the Theorem below.

\begin{theorem}\label{thm-71}
		Let $q\geq1$ and $W_\cA$ be a shift-invertible distribution (see Definition \ref{defShiftInv} below). Assume that, for every $1\leq p\leq\infty$, $\hat S:L^p(\rd)\to L^p(\rd)$ (see \eqref{Shift-inv} below) is a homeomorphism. Let $a\in\cS'(\rdd)$ and $Op_\cA(a)$ be the associated metaplectic pseudofifferential operator (see \eqref{defOPA} below). Then, the mapping $Op_\cA(a)\in BL(L^p(\rd))$ if and only if $q\leq 2$ and $q\leq p\leq q'$.
\end{theorem}

Theorem \ref{thm-71} is a partial result for at least two reasons: it focuses on metaplectic operators that are homeomorphisms of $L^p$, and even under this restriction, it leaves unresolved the question of precisely characterizing these operators within this context. \\

For the purposes of the present work, we need a brief digression on metaplectic Wigner distributions. Apart from its  importance in the context of time-frequency analysis, this discussion provides some example of how metaplectic operators are applied to define quantizations for pseudodifferential operators and how their properties are related to the structure of the blocks of their projections. 

Consider a metaplectic operator on $L^2(\rdd)$, denoted by $\hat\cA\in\Mp(2d,\bR)$. The associated \emph{metaplectic Wigner distribution} $W_\cA$ is the time-frequency representation defined for every $f,g\in\cS'(\rd)$ as
\begin{equation}\label{defWA}
	W_\cA(f,g)=\hat\cA(f\otimes\bar g).
\end{equation}
For a comprehensive treatment of these distributions and their properties in terms of the symplectic projections of the corresponding metaplectic operators $\hat\cA$, we refer the reader to \cite{cordero2024metaplectic, cordero2024excursus, cordero2024unified}. Notably, we stress that $W_\cA:\cS(\rd)\times\cS(\rd)\to\cS(\rdd)$ is bounded, $W_\cA:L^2(\rd)\times L^2(\rd)\to L^2(\rdd)$ is bounded and $W_\cA:\cS'(\rd)\times \cS'(\rd)\to \cS'(\rdd)$ is bounded.

Metaplectic Wigner distributions are natural generalizations of the classical cross-Wigner distribution, defined in \eqref{intro5} for $L^2$ functions, which reads as:
\[
	W(f,g)(x,\xi)=\hat A_{1/2}(f\otimes \bar g),
\]
where the projection $A_{1/2}\in\Sp(2d,\bR)$ has $d\times d$ block decomposition:
\[
	A_{1/2}=\begin{pmatrix}
		I_d/2 & I_d/2 & 0_d & 0_d\\
		0_d & 0_d & I_d/2 & -I_d/2\\
		0_d & 0_d & I_d & I_d\\
		-I_d & I_d & 0_d & 0_d
	\end{pmatrix}.
\]
Metaplectic Wigner distributions provide quantization laws for pseudodifferential operators. Specifically, for given $a\in\cS'(\rdd)$ (symbol) and $W_\cA$ metaplectic Wigner distribution (quantization), the operator $Op_\cA(a):\cS(\rd)\to\cS'(\rd)$ defined by:
\begin{equation}\label{defOPA}
	\la Op_\cA(a)f,g\ra = \la a,W_\cA(g,f)\ra, \qquad f,g\in\cS(\rd),
\end{equation}
is the \emph{pseudodifferential operator} with symbol $a$ and quantization $W_\cA$.

In \cite{cordero2024unified} the authors proved Theorem \ref{thm-71}, a boundedness result for pseudodifferential operators with symbols in Lebesgue spaces having as quantization a \emph{shift-invertible} metaplectic Wigner distribution. These time-frequency representations play a central role in characterizing the quasi-norms of modulation spaces.

\begin{definition}\label{defShiftInv}
	A metaplectic Wigner distribution $W_\cA$ is \emph{shift-invertible} or, equivalently, $\hat\cA$ and $\cA=\pi^{Mp}(\hat \cA)$ are shift-invertible, if there exist $L\in\GL(d,\bR)$, $C\in\Sym(d,\bR)$ and $\hat S\in\Mp(d,\bR)$ such that:
	\begin{equation}\label{Shift-inv}
	W_\cA(f,g)(z)=|\det(L)|^{1/2}\Phi_C(Lz)W(f,\hat S g)(Lz), \qquad f,g\in L^2(\rd), \ z\in\rdd.
	\end{equation}
	We denote by $\Shp(2d,\bR)=\{\cA\in\Sp(2d,\bR) \ \text{shift-invertible}\}$, the set of shift-invertible matrices. 
\end{definition}

The original definition of shift-invertibility was given in \cite{cordero2023characterization} in terms of the invertibility of the submatrix
 	\[
			E_\cA=\begin{pmatrix}
				A_{11} & A_{13} \\
				A_{21} & A_{23} 
			\end{pmatrix}
	\]
 of the projection
\begin{equation}\label{blockSA}
			\cA=\begin{pmatrix}
				A_{11} & A_{12} & A_{13} & A_{14}\\
				A_{21} & A_{22} & A_{23} & A_{24}\\
				A_{31} & A_{32} & A_{33} & A_{34}\\
				A_{41} & A_{42} & A_{43} & A_{44}
			\end{pmatrix}, \qquad A_{ij}\in\bR^{d\times d}, \ i,j=1,\ldots,4.
\end{equation}
Under the notation above, it was proved in \cite{cordero2024metaplectic, cordero2024unified} that $E_\cA\in\GL(d,\bR)$ if and only if $W_\cA$ enjoys the representation formula \eqref{Shift-inv}. 

\begin{remark}\label{rem-Wign}
	If $W_\cA$ is shift-invertible, with expression \eqref{Shift-inv}, and $a\in\cS'(\rdd)$, we can relate $Op_\cA(a)$ to a pseudodifferential operator quantized by the cross-Wigner distribution. Indeed,
	\begin{align*}
		\la Op_\cA(a)f,g\ra &= \la a,W_\cA(g,f)\ra\\
		&=\la \mathfrak{p}_{-C}\mathfrak{T}_{L}^{-1}a,W(g,\hat Sf\ra\\
		&=\la \tilde a^w(x,D)\hat S f,g\rangle,
	\end{align*}
	where $\tilde a=\mathfrak{p}_{-C}\mathfrak{T}_{L}^{-1}a$ and
	\begin{equation}\label{exppseudo}
		\tilde a^w(x,D)f(x)=\int_{\rdd}e^{2\pi i(x-y)\xi}\tilde a\Big(\frac{x+y}{2},\xi\Big)f(y)dyd\xi, \qquad f\in\cS(\rd).
	\end{equation}
	(see also \cite{cordero2022wigner2}). Observe that the metaplectic operator $\hat{\tilde S}=\mathfrak{p}_{-C}\mathfrak{T}_{L}^{-1}$ is a homeomorphism of $L^q(\rdd)$ for every $0<q\leq\infty$, so that $a\in L^q(\rdd)$ if and only if $\tilde a\in L^q(\rdd)$, with
	\[
		\norm{\tilde a}_q=C_{q,\tilde S}\norm{a}_q.
	\]
\end{remark}

We use the following result from \cite{boggiatto2009weyl}, originally stated for the short-time Fourier transform, and formulated hereafter for the cross-Wigner distribution \eqref{intro5}. 

\begin{proposition}\label{prop-Bogg1}
Let $1\leq p,q\leq\infty$.\\
	(i) If $q\geq2$ and $q'\leq p\leq q$, $W:L^{p'}(\rd)\times L^p(\rd)\to L^q(\rdd)$ is bounded.\\
	(ii) If $p>q$ or $p<q'$, then $W$ is not bounded from $L^{p'}(\rd)\times L^p(\rd)$ to $L^q(\rdd)$.
\end{proposition}
\begin{proof}
	It is a restatement of \cite[Propositions 3.1 and 3.2]{boggiatto2009weyl} using \eqref{relSTFTW}.
	
\end{proof}

Proposition \ref{prop-Bogg1} was generalized to metaplectic Wigner distributions \eqref{Shift-inv} with $\hat S\in B(L^p(\rd))$ in \cite[Proposition 3.6]{cordero2024unified}. The following result is a further generalization to the case in which $\hat S$ is either free or a homeomorphism of $L^p(\rd)$.

\begin{theorem}\label{thmWAb}
Let $W_\cA$ be a shift-invertible metaplectic Wigner distribution, as in \eqref{Shift-inv}, and $1\leq p,q\leq\infty$. Let $S=\pi^{Mp}(\hat S)$ have block decomposition \eqref{blockS}. \\
	(i) If $\hat S\in B(L^p(\rd))$, then $W_\cA:(f,g)\in L^{p'}(\rd)\times L^p(\rd)\to W_\cA(f,g)\in L^q(\rdd)$ is bounded if and only if $q\geq2$ and $q'\leq p\leq q$.\\
	(ii) If $\hat S$ is free and $q\geq2$ and $q'\leq p\leq q$, then $W_\cA:(f,g)\in L^{p}(\rd)\times L^p(\rd)\to W_\cA(f,g)\in L^q(\rdd)$ is bounded.\\
\end{theorem}
\begin{proof}
	Item $(i)$ is the content of \cite[Proposition 3.6]{cordero2024unified}. We prove $(ii)$ with a similar argument. For every $f,g\in L^2(\rd)$, Moyal's identity:
	\begin{align*}
		\norm{W_\cA(f,g)}_2^2&=\la W_\cA(f,g),W_\cA(f,g)\ra\\
		&=\la \hat \cA(f\otimes\overline{g}),\hat\cA(f\otimes\overline{g})\ra\\
		&=\la f\otimes\overline{g},f\otimes\overline{g}\ra\\
		&=\norm{f}_2^2\norm{g}_2^2.
	\end{align*}
	tells that $\norm{W_\cA(f,g)}_2=\norm{f}_2\norm{g}_2$. On the other hand, from \eqref{intro5} and H\"older's inequality, for every $1\leq r\leq2$,
	\begin{align*}
		\norm{W_\cA(f ,g)}_\infty&=|\det(L)|^{1/2}\norm{W(f,\hat Sg)(L\cdot)}_\infty\\
		&=|\det(L)|^{1/2}\norm{W(f,\hat Sg)}_\infty\\
		&\lesssim\norm{f}_{r}\norm{\hat S g}_{r'}\\
		&\lesssim\norm{f}_{r}\norm{g}_{r}.
	\end{align*}
	for every $f,g\in L^2(\rd)$ and every $x,\xi\in\rd$. Therefore, $\norm{W_\cA(f,\hat Sg)}_\infty\leq C_{S,r}\norm{f}_r\norm{g}_r$ for every $1\leq r\leq2$. By multilinear interpolation (see e.g. \cite[Theorem 2.7]{bennett1988interpolation}), $W_\cA:L^p(\rd)\times L^p(\rd)\to L^q(\rdd)$ is bounded for every $q\geq2$ and $q'\leq p\leq q$. 

\end{proof}

We are ready to formulate an improved version of Theorem \ref{thm-71}.

\begin{theorem}\label{final-thm-appl}
	Let $1\leq p,q\leq\infty$. Let $W_\cA$ be a shift-invertible metaplectic Wigner distribution as in \eqref{Shift-inv} with $S=\pi^{Mp}(\hat S)$. Let $a\in L^q(\rdd)$ and $Op_\cA(a):\cS(\rd)\to\cS'(\rd)$ be the associated metaplectic operator. The following statements hold true.\\
	(i) If $\hat S=\mathfrak{p}_Q\mathfrak{T}_L$ for some $Q\in\Sym(d,\bR)$ and $L\in\GL(d,\bR)$, then $Op_\cA(a)\in B(L^p(\rd))$ if and only if $1\leq q\leq 2$ and $q\leq p\leq q'$.\\
	(ii) If $\hat S$ is free, $1\leq q\leq 2$ and $q\leq p\leq q'$, then $Op_\cA(a)\in B(L^p(\rd), L^{p'}(\rd))$.
\end{theorem}
\begin{proof}
	$(i)$ is a restatement of \cite[Theorem 3.8]{cordero2024unified} using Theorem \ref{main-thm-ref}. We prove $(ii)$ with an elementary duality argument for $(p,q)\neq(+\infty,1)$, and using Remark \ref{rem-Wign} for the remaining case. For $f,g\in\cS(\rd)$, $1\leq q\leq2$ and $q\leq p\leq q'$, by Theorem \ref{thmWAb} $(ii)$, 
	\begin{align*}
		|\la Op_\cA(a)f,g\ra|&=|\la a,W_\cA(g,f)\ra|\leq\norm{a}_q\norm{W_\cA(g,f)}_{q'}\lesssim\norm{a}_q\norm{f}_p\norm{g}_p,
	\end{align*}
	since $q'\geq2$. A standard density argument entails that:
	\[
		\norm{Op_\cA(a)f}_{p'}\lesssim\norm{a}_q\norm{f}_p, \qquad f\in L^p(\rd)
	\]
	and, consequently, that $Op_\cA(a)\in B(L^p(\rd),L^{p'}(\rd))$.
	
	This proves the assertion for all the couples $(p,q)$ with $(p,q)\neq(+\infty,1)$. By \eqref{exppseudo}, 
	\begin{align*}
		\norm{Op_\cA(a)f}_\infty&\leq\int_{\rdd}\Big|\tilde a\Big(\frac{x+y}{2},\xi\Big)\Big||\hat Sf(y)|dyd\xi\\
		&\lesssim\norm{\hat S f}_\infty\int_{\rdd}|\tilde a(z,\xi)|dzd\xi\\
		&=\norm{\tilde a}_1\norm{\hat Sf}_\infty\\
		&\lesssim\norm{a}_1\norm{f}_1,
	\end{align*}
	where $\tilde a$ is defined as in Remark \ref{rem-Wign}. This concludes the proof for $(p,q)=(+\infty,1)$, thereby completing the proof of the theorem.
	
\end{proof}

\section{Representation formulae for $Mp(2d,\bR)$ and applications to time-frequency analysis}\label{sec:attfan}
	Among all the metaplectic Wigner distributions, shift-invertible distributions revealed to play a central role in the measurement of local time-frequency content of signals \cite{cordero2023symplectic,cordero2024metaplectic}. As a first contribution of this section, we employ the Dopico-Johnson factorization to demonstrate the density of shift-invertible matrices in $\Sp(2d,\mathbb{R})$. As a consequential insight, we unveil a factorization theorem asserting that any $4d\times4d$ symplectic matrix can be decomposed (albeit non-uniquely) as the product of a free symplectic matrix and a shift-invertible symplectic matrix. 
		
	\begin{lemma}\label{lemmaS-I}
		Let $\hat\cA\in\Mp(2d,\bR)$. The following are equivalent.\\
		(i) $\hat\cA$ is shift-invertible.\\
		(ii) Any Dopico-Johnson factorization $\cA=V_Q\cD_L V_P^T\Pi_\cJ$ of the symplectic projection of $\hat\cA$, with $P,Q\in\Sym(2d,\bR)$, $L\in\GL(2d,\bR)$ and $\cJ\subseteq\{1,\ldots,2d\}$, has $P_{12}\in\GL(d,\bR)$, where:
		\begin{equation}\label{blockP}
			P=\begin{pmatrix}
				P_{11} & P_{12}\\
				P_{12}^T & P_{22}
			\end{pmatrix}, \qquad P_{11},P_{12},P_{22}\in\bR^{d\times d}, \ P_{11}=P_{11}^T, \ P_{22}=P_{22}^T.
		\end{equation}
	\end{lemma} 
	\begin{proof}
		Let $\cA=V_Q\cD_L V_P^T\Pi_\cJ$ be any Dopico-Johnson factorization of $\cA$, as in item $(ii)$, and write $\cJ=\cJ_1\cup\cJ_2$, with $\cJ_1\subseteq\{1,\ldots,d\}$, $\cJ_2\subseteq\{d+1,\ldots,2d\}$. Then, 
		\[
			\Pi_\cJ=\left(\begin{array}{ c c | c c}
				I_{\cJ_1^c} & 0 & I_{\cJ_1} & 0\\
				0 & I_{\cJ_2^c} & 0 & I_{\cJ_2}\\
				\hline
				-I_{\cJ_1} & 0 & I_{\cJ_1^c} & 0\\
				0 & -I_{\cJ_2} & 0 & I_{\cJ_2^c}
			\end{array}\right)
		\]
		(see e.g., \cite[Appendix B]{cordero2023symplectic}). A straightforward computation shows that:
		\[
			\cA=V_Q\cD_L\left(\begin{array}{c c | c c}
				 
				 I_{\cJ_1^c}-P_{11} I_{\cJ_1} & -P_{12}I_{\cJ_2} & I_{\cJ_1}+P_{11}I_{\cJ_1^c} & P_{12}I_{\cJ_2}\\
				 -P_{12}^TI_{\cJ_1} & I_{\cJ_2^c}-P_{22}I_{\cJ_2} & P_{12}^TI_{\cJ_1^c} & I_{\cJ_2^c}+P_{22}I_{\cJ_2}\\
				 \hline
				 -I_{\cJ_1} & 0 & I_{\cJ_1^c} & 0\\
				 0 & -I_{\cJ_2} & 0 & I_{\cJ_2^c}
			\end{array}\right).
		\]
		The product by $\cD_L$ modifies the upper blocks with a left multiplication by $L^{-1}$, whereas the product by $V_Q$ does not affect the upper blocks. Therefore,
		\[
			E_\cA=L^{-1}\begin{pmatrix}
				I & P_{11} \\
				0 & P_{12}^T
			\end{pmatrix}
			\begin{pmatrix}
				I_{\cJ_1^c} & I_{\cJ_1}\\
				-I_{\cJ_1} & I_{\cJ_1^c}
			\end{pmatrix}=L^{-1}\begin{pmatrix}
				I & P_{11} \\
				0 & P_{12}^T
			\end{pmatrix}\Pi_{\cJ_1}.
		\]
		Hence, $E_\cA\in\GL(2d,\bR)$ if and only if $P_{12}\in \GL(d,\bR)$, and we are done.
		
	\end{proof}
	
	Lemma \eqref{lemmaS-I} is interesting on its own, as it entails a decomposition law for $\Mp(2d,\bR)$, improving the well-known factorization of symplectic matrices in terms of free symplectic matrices (see e.g. \cite{folland1989harmonic}), and the density of $\Shp(2d,\bR)$ in $\Sp(2d,\bR)$.
	\begin{theorem}\label{thm-shift-inv-constr}
	The following statements hold true.\\
	(i) $\Shp(2d,\bR)\subseteq\Sp(2d,\bR)$ is dense. \\
	(ii) For every $\cA\in \Sp(2d,\bR)$ there exists $\cA'\in\Shp(2d,\bR)$ and $\Xi\in\Sp(2d,\bR)$ free such that $\cA=\Xi\cA'$.\\
	(iii) For every $\cA\in \Sp(2d,\bR)$ there exists $\cA''\in\Shp(2d,\bR)$ and $\Theta\in\Sp(2d,\bR)$ such that $\cA=\cA''\Theta$, where the block decomposition \eqref{blockS} of $\Theta$ has $\det(A)\neq0$.
	\end{theorem}
	\begin{proof}
		Let $\cA=V_Q\cD_LV_P^T\Pi_\cJ$ be a Dopico-Johnson factorization of $\cA$, with $P\in\Sym(2d,\bR)$ having block decomposition \eqref{blockP}. For every $0<\tau<\min\{|\lambda| : \lambda\in\mbox{eig}(P_{12})\setminus\{0\}\}$ (where the minimum is $+\infty$ if $P_{12}=0$) $P_{12}+\tau I$ is invertible. It follows that the matrix $\cA_\tau=V_Q\cD_LV_{P+\tau R}\Pi_\cJ$ is shift-invertible, where:
		\[
			R=\begin{pmatrix}
				0 & I\\
				I & 0
			\end{pmatrix},
		\]
		and it can be easily related to $\cA$:
		\begin{align}
			\nonumber
			\cA_\tau&=V_Q\cD_LV_{P+\tau R}^T\Pi_\cJ\\
			\label{number}
			&=V_Q\cD_LV_{\tau R}^TV_P^T\Pi_\cJ\\
			\nonumber
			&=V_QV^T_{\tau L^{-1}RL^{-T}}\cD_LV_P^T\Pi_\cJ\\
			\nonumber
			&=V_QV^T_{\tau L^{-1}RL^{-T}}V_{-Q}V_Q\cD_LV_P^T\Pi_\cJ\\
			\nonumber
			&=V_QV^T_{\tau L^{-1}RL^{-T}}V_{-Q}\cA.
		\end{align}
		The matrices
		\[
			\Xi_\tau=V_QV^T_{\tau L^{-1}RL^{-T}}V_{-Q} =\begin{pmatrix}
				I-\tau L^{-1}RL^{-T}Q & \tau L^{-1}RL^{-T}\\
				-\tau QL^{-1}RL^{-T}Q & I+\tau QL^{-1}RL^{-T}
			\end{pmatrix}
		\]
		are obviously free for every $\tau\neq0$ sufficiently small. Observe that this is equivalent to having $\Xi_\tau^{-1}$ free for every $\tau\neq0$ sufficiently small. Moreover, $\Xi_\tau\to I_{2d}$ for $\tau\to0$ component-wise. 
		Consequently, $\cA=\Xi_\tau^{-1}\cA_\tau$ for every $\tau\neq0$ sufficiently small, with $\cA_\tau$ shift-invertible and $\Xi_\tau^{-1}$ free. This concludes the proof of $(i)$ and $(ii)$. To prove $(iii)$, we continue the computation above from \eqref{number}. Using Lemma \ref{lemma-redox},
	\begin{align*}
		\cA_\tau&=V_Q\cD_LV_P^TV_{\tau R}^T\Pi_\cJ\\
		&=V_Q\cD_LV_P^T\Pi_\cJ(\Pi_\cJ^{-1}V_{\tau R}^T\Pi_\cJ)\\
		&=\cA(V^T_{\tau I_{\cJ^c}RI_{\cJ^c}}\cD_{I+\tau I_{\cJ^c}RI_{\cJ}}V_{-\tau I_{\cJ}RI_{\cJ}}),
	\end{align*}
	where:
	\begin{align*}
		V^T_{\tau I_{\cJ^c}RI_{\cJ^c}}\cD_{I+\tau I_{\cJ^c}RI_{\cJ}}V_{-\tau I_{\cJ}RI_{\cJ}}=\begin{pmatrix}
			I-\tau I_{\cJ^c}RI_{\cJ} & \tau I_{\cJ^c}RI_{\cJ^c}\\
			-\tau I_{\cJ}RI_{\cJ} & I+\tau I_{\cJ}RI_{\cJ^c}
		\end{pmatrix}.
	\end{align*}		
	Since $I+\tau I_{\cJ^c}RI_{\cJ}$ is the inverse of $I-\tau I_{\cJ^c}RI_{\cJ} $, $(iii)$ follows, and we are done.
	
	\end{proof}
	
	Lemma \ref{lemmaS-I} states that shift-invertibility can be expressed in terms of Dopico-Johnson factorizations as the condition $P_{12}\in \GL(d,\bR)$ stated in \eqref{blockP}. In \cite{cordero2024metaplectic}, it was proved that \emph{if} $W_\cA$ is shift-invertible \emph{then} for every $g\in\cS(\rd)\setminus\{0\}$,
	\begin{equation}\label{equivNorms}
		\norm{f}_{M^p}\asymp\norm{W_\cA(f,g)}_{{p}}, \qquad f\in M^p(\rd),
	\end{equation}
	showing that shift-invertibility is a fundamental property in measuring local time-frequency content. 
	
	\begin{remark} We outlined the (quasi-)norm equivalence \eqref{equivNorms} for every $0<p\leq\infty$, but it is far more general. It applies to every \emph{weighted} modulation space $M^{p,q}_m(\rd)$ ($0<p,q\leq\infty$, $m$ moderate weight function, see e.g., \cite{cordero2020time} for the more general definition of $M^{p,q}_m$), under the further assumption $A_{21}=0$ in \eqref{blockSA}. In the same work, it was observed that if either shift-invertiblity does not hold, or $A_{21}\neq0$, then there exists $W_\cA$ such that:
	\[
		\norm{f}_{M^{p,q}_m}\not\asymp\norm{W_\cA(f,g)}_{L^{p,q}_m}, \qquad f\in M^{p,q}_m(\rd)
	\] 
	($L^{p,q}_m$ are the mixed-norm Lebesgue spaces). \end{remark} 
	
	Here, we prove rigorously that \emph{if} $W_\cA$ is not shift-invertible, \emph{then} \eqref{equivNorms} fails for every $0<p\leq\infty$. 
	
	\begin{theorem}
		Let $0<p\leq\infty$. Let $W_\cA$ be a non shift-invertible metaplectic Wigner distribution. Then, there exists $f\in M^p(\rd)$ such that $\norm{W_\cA(f,g)}_{p}\not\asymp\norm{f}_{M^p}$.
	\end{theorem}
	\begin{proof}
		Let $\cA=\cD_L V_QV_P^T\Pi_\cJ$ be a Dopico-Johnson factorization of $\cA$, with $L\in\GL(2d,\bR)$, $P,Q\in\Sym(2d,\bR)$ and $\cJ\subseteq\{1,\ldots,2d\}$. Since $\cA$ is not shift-invertible, $P_{12}\notin\GL(d,\bR)$ by Lemma \ref{lemmaS-I}. Let us write $\cJ=\cJ_1\cup\cJ_2$, with $\cJ_1\subseteq\{1,\ldots,d\}$ and $\cJ_2\subseteq\{d+1,\ldots,2d\}$, and
		\[
		\tilde P =\begin{pmatrix}
			P_{11} & 0\\
			0 & P_{22}
		\end{pmatrix}, \qquad \tilde Q = \begin{pmatrix}
			Q_{11} & 0\\
			0 & Q_{22}
		\end{pmatrix},
		\]
		where $P_{ij},Q_{ij}$, $i,j=1,\ldots,2$ are the $d\times d$ blocks of $P$ and $Q$, respectively. Let
		\begin{equation}\label{defM}
			M=\begin{pmatrix}
				I+P_{12}Q_{12}^T & -P_{12}\\
				-Q_{12}^T & I
			\end{pmatrix}\in\GL(2d,\bR).
		\end{equation}
		It is easy to show that:
		\begin{equation}\label{Adecompthm}
			\cA=\cD_LV_{\tilde Q}\cA_{FT2}^{-1}\cD_M\cA_{FT2}V^T_{\tilde P}\Pi_{\cJ_1}\Pi_{\cJ_2},
		\end{equation}
		where $\cA_{FT2}\in\Sp(2d,\bR)$ is the symplectic projection of the partial Fourier transform with respect to the frequency variables, $\cF_2=\cF_{\{d+1,\ldots,2d\}}$, i.e.,
		\[
		\cA_{FT2}=\begin{pmatrix}
			I_d & 0_d & 0_d & 0_d\\
			0_d & 0_d & 0_d & I_d\\
			0_d & 0_d & I_d & 0_d\\
			0_d & -I_d & 0_d & 0_d
		\end{pmatrix}.
		\]
		
		From the metaplectic Wigner distributions perspective, \eqref{Adecompthm} reads as:
		\[
			W_\cA(f,g)=(\mathfrak{T}_L\mathfrak{p}_{\tilde Q}\cI_2)(\cF_2\mathfrak{T}_M)(\mathfrak{m}_{P_{11}}\cF_{\cJ_1}f\otimes\cF\mathfrak{m}_{P_{22}}\cF_{\cJ_2}\bar g)
		\]
		(up to a sign) where $\cI_2F(x,\xi)=F(x,-\xi)$. The operator $\mathfrak{T}_L\mathfrak{p}_{\tilde Q}\cI_2$ is a homeomorphims of $L^p(\rdd)$, whereas metaplectic operators are homeomorphisms of $M^p(\rd)$, as observed in \cite{fuhr2024metaplectic}. Therefore, it is enough to prove the assertion for $\cF_2\mathfrak{T}_M$ instead of $\hat\cA$. 

		The following integrals must be interpreted in the sense of distributions. Let $\f\in\cS'(\rd)$ and $\psi\in\cS(\rd)\setminus\{0\}$ to be fixed. Using the change of variables $-Q_{12}^Tx+t=s$, we have:
		\begin{align*}
			\cF_2\mathfrak{T}_M(\f\otimes\psi)(x,\xi)&=\int_{\rd}\f(x+P_{12}Q_{12}^Tx-P_{12}t)\overline{\psi(-Q_{12}^Tx+t)}e^{-2\pi i\xi t}dt\\
			&=\int_{\rd}\f(x-P_{12}s)\overline{\psi(s)}e^{-2\pi i\xi(s+Q_{12}^Tx)}ds\\
			&=e^{-2\pi i\xi Q_{12}^Tx}\int_{\rd}\f(x-P_{12}s)\overline{\psi(s)}e^{-2\pi i\xi s}ds.
		\end{align*}
		Since $P_{12}$ is symmetric, $P_{12}=\Sigma^T\Delta\Sigma$, where $\Sigma\in\bR^{d\times d}$ is orthogonal and $\Delta$ is the corresponding diagonal matrix with diagonal entries given by the eigenvalues $\lambda_1,\ldots,\lambda_{d}$ of $P_{12}$. By Lemma \ref{lemmaS-I}, $P_{12}\notin\GL(2d,\bR)$ so that the set $\cJ=\{j:\lambda_j=0\}$ is non-empty. Without loss of generality, let us assume that $\lambda_1,\ldots,\lambda_r\neq0$ and $\lambda_{r+1},\ldots,\lambda_d=0$. Choosing $\psi(t)=e^{-\pi|\cdot|^2}$ (with an abuse of notation, hereafter $\psi$ denotes a Gaussian on $\bR^n$ for any $n$), using the change of variables $\Sigma s=u$, and denoting $(y,\eta)=(\Sigma x,\Sigma\xi)$, 
		\begin{align*}
		\cF_2\mathfrak{T}_M(\f\otimes\bar\psi)(x,\xi)&=e^{-2\pi i\xi Q_{12}^Tx}\int_{\rd}\f(x-\Sigma^T\Delta u)\overline{\psi(\Sigma^Tu)}e^{-2\pi i\xi \Sigma^Tu}du\\
		&=e^{-2\pi i\xi Q_{12}^Tx}\int_{\rd}\f\circ\Sigma^T(y(x)-\Delta u){\psi(u)}e^{-2\pi i\Sigma\eta(x) u}du.
		\end{align*}
		Let $\tilde\f=\f\circ\Sigma^T$ (observe that rescalings preserve modulation spaces), then
		\begin{align*}
		\cF_2\mathfrak{T}_M(\f\otimes\bar\psi)(x,\xi)&=e^{-2\pi i\xi Q_{12}^Tx}\int_{\rd}\tilde\f(y(x)-\Delta u){\psi(u)}e^{-2\pi i\eta(\xi) u}du.
		\end{align*}
		Choosing $\tilde{\f}(v)=\tilde{\f}_1(v_{\cJ^c})\tilde{\f}_2(v_\cJ)$, we find:
		\begin{equation}\begin{split}\label{split2}
		\cF_2\mathfrak{T}_M&(\f\otimes\bar\psi)(x,\xi)=e^{-2\pi i\xi Q_{12}^Tx}V_{\psi}\tilde \f_1((\Sigma x)_{\cJ},(\Sigma\xi)_{\cJ})\tilde\f_2((\Sigma x)_{\cJ^c})\psi((\Sigma \xi)_{\cJ^c})
		\end{split}\end{equation}
		Since the change of variables $(y,\eta)=(\Sigma x,\Sigma\xi)$ preserves the $L^p$ (quasi-)norms,
		\begin{equation}\label{STFTsplit}
			\norm{V_\psi\f}_p\asymp\norm{\tilde\f_1}_{M^p}\norm{\tilde\f_2}_{p}\norm{\psi}_p,
		\end{equation}
		and, since $\norm{\tilde\f_2}_{M^p}\not\asymp\norm{\tilde\f_2}_{p}$, the assertion follows.

	\end{proof}

	\begin{remark} A rescaling $\mathfrak{T}_LF$ is bounded on (a homeomorphism of) the mixed-norm Lebesgue spaces $L^{p,q}(\rdd)$ ($p\neq q$) if and only if $L\in\GL(2d,\bR)$ has block decomposition 
	\[
		L=\begin{pmatrix} A & B \\ 0 & D\end{pmatrix}, \qquad A,B,D\in\bR^{d\times d}.
	\]
	However, it is straightforward to verify that the operator $\mathfrak{T}_M$, with $M$ defined as in \eqref{defM}, possesses the following boundedness property on tensor products:
	\begin{equation}\label{conj}
		\norm{\mathfrak{T}_M(f\otimes g)}_{L^{p,q}}=\norm{f}_p\norm{g}_q.
	\end{equation}
	Although we will not elaborate on this fact here, we conjecture that \eqref{conj} can be utilized to study the boundedness of metaplectic Wigner distributions in greater detail. This could potentially used to characterize $\norm{W_\cA(f,g)}_{L^{p,q}}$ ( $0<p,q\leq\infty$) in terms of the function spaces to which $f$ belongs.
	\end{remark}
	
%

\section*{Acknowledgements}
	The author is financially supported by the University of Bologna and HES-SO Valais-Wallis. The author thanks professors Nicola Arcozzi, Elena Cordero and Luigi Rodino for reading the manuscript and providing crucial suggestions.

\bibliographystyle{plain}
\bibliography{BiblioMetapLp.bib}

\end{document}